\documentclass[10pt,oneside,leqno]{amsart}
\usepackage{amsxtra}
\usepackage{amsopn}
\usepackage{color}
\usepackage{amsmath,amsthm,amssymb,tabularx}
\usepackage{amscd}
\usepackage{enumitem}
\usepackage{amsfonts}
\usepackage{latexsym}
\usepackage{verbatim}
\usepackage{graphicx}
\usepackage{graphics}
\usepackage{multirow}
\usepackage{lscape}
\usepackage{scrextend}
\usepackage{mathabx}
\usepackage{nicefrac}
\usepackage{tikz}

\usepackage[hypertexnames=false,
 backref=UseNone,
    pdftex,
    pdfpagemode=UseNone,
    breaklinks=true,
    extension=pdf,
    colorlinks=true,
    linkcolor=blue,
    citecolor=blue,
    urlcolor=blue,
]
{hyperref}
\theoremstyle{plain}
\newtheorem{theorem}{Theorem}[section]

\newtheorem{lemma}[theorem]{Lemma}
\newtheorem{proposition}[theorem]{Proposition}
\newtheorem{corollary}[theorem]{Corollary}
\newtheorem{remark}[theorem]{Remark}

\newtheorem{example}[theorem]{Example}

\newtheorem{remark-question}[section]{Remark-Question}

\newcommand{\Gtwo}{\mathrm{G}_2}

\newcommand\R{{\mathbb R}}

\newcommand\frg{{\mathfrak g}}

\newcommand\frn{{\mathfrak n}}

\newcommand\frs{{\mathfrak s}}

\definecolor{fondo}{rgb}{0.93,0.93,0.93}

\definecolor{m}{rgb}{0.9,0,0.9}

\makeatletter
\renewcommand*{\eqref}[1]{%
  \hyperref[{#1}]{\textup{\tagform@{\ref*{#1}}}}%
}
\makeatother
\begin{document}
\title[Solutions of the Laplacian flow and coflow of a LCP $\mathrm{G}_2$-structure]{Solutions of the Laplacian flow and coflow of a Locally Conformal Parallel $\mathrm{G}_2$-structure}

 \author{Victor Manero}
 \address[V. Manero]{Departamento de Matem\'aticas\,-\,I.U.M.A.\\
 Universidad de Zaragoza\\
 Facultad de Ciencias Humanas y de la Educaci\'on\\
 22003 Huesca, Spain}
 \email{vmanero@unizar.es}

 \author{Antonio Otal}
 \address[A. Otal and R. Villacampa]{Centro Universitario de la Defensa\,-\,I.U.M.A., Academia General
 Mili\-tar, Crta. de Huesca s/n. 50090 Zaragoza, Spain}
 \email{aotal@unizar.es}
 \email{raquelvg@unizar.es}

 \author{Raquel Villacampa}

\date{\today}

\maketitle

\begin{abstract}
%\textcolor{red}{We present the Laplacian Locally Conformal Parallel flow of $\mathrm{G}_2$-structures: $\frac{d}{dt}\sigma=\Delta_t \sigma$ where $\sigma$ is the 3-form of a Locally Conformal Parallel $\mathrm{G}_2$-structure and $\Delta_t$ is the Hodge Laplacian operator.}
We study the Laplacian flow of a $\mathrm{G}_2$-structure where this latter structure is claimed to be Locally Conformal Parallel. The first examples of long time solutions of this flow with the Locally Conformal Parallel condition are given. All of the solutions are ancient and Laplacian soliton of shrinking type. These examples are
one-parameter families of Locally Conformal Parallel $\mathrm{G}_2$-structures on rank-one solvable extensions of six-dimensional nilpotent Lie groups.  The  found solutions are used to construct  long time  solutions to the Laplacian coflow starting from a Locally Conformal Parallel structure.
We also study the behavior of the curvature of the solutions obtaining that for one of the examples the induced metric is Einstein along all the flow (resp. coflow). %In contrast, for the rest of examples the induced metrics are not even Ricci-soliton.
\end{abstract}

\setcounter{tocdepth}{3} \tableofcontents

\bigskip

%\footnote{Rq 23 agosto: Comentario ``profundo'': realmente no se si estamos definiendo un nuevo flujo (me refiero a la ecuación puramente diferencial) o simplemente cambiando una de las condiciones.  Lo digo por la terminología a utilizar (LCP-flow o como dice el titulo LF para LCP structures) y en general por la estructura del paper.  Tras haber escuchado algunas charlas aqui, me decanto mas por la opción del titulo y por no definir un nuevo flujo, ya que no estudiamos ninguna de las cuestiones típicas: parabolicidad, existencia de solución a tiempo corto,... ni siquiera las mencionamos.  Creo que podriamos indicar que encontramos soluciones al flujo laplaciano dentro de otra clase de Fernandez-Gray y a modo de conclusi\'on, que ``sugerimos'' (o que estas soluciones nos sugieren) el estudio del flujo LCP.}

%%%%%%%%%%%%%%%%%%%%%%%%%%%%%%%%%%%%%%%%%%%%%%%%%%%%%%%%%%%%%%%%%%%%%%%%%%%%%%%%%%%%%%
%INTRODUCTION%%%%%%%%%%%%%%%%%%%%%%%%%%%%%%%%%%%%%%%%%%%%%%%%%%%%%%%%%%%%%%%%%%%%%%%%%%%%%
%%%%%%%%%%%%%%%%%%%%%%%%%%%%%%%%%%%%%%%%%%%%%%%%%%%%%%%%%%%%%%%%%%%%%%%%%%%%%%%%%%%%%%

\begin{section}*{Introduction}

A $\mathrm{G}_2$-structure on a $7$-dimensional smooth manifold $M$ is a reduction  to the exceptional Lie group $\mathrm{G}_2$ of the structure group $\textrm{GL}(7,\R)$
of the frame bundle of $M$. We call {\em $\mathrm{G}_2$-manifold} a 7-dimensional manifold endowed with a $\mathrm{G}_2$-structure. The presence of a $\mathrm{G}_2$-structure is equivalent to the existence
of a globally defined 3-form $\varphi$,
which is called  the  \emph{$\mathrm{G}_2$ form}  or the \emph{fundamental  $3$-form}
and it can be described locally as
\begin{equation}\label{sigma}
 \varphi=e^{127}+e^{347}+e^{567}+e^{135}-e^{146}-e^{236}-e^{245},
\end{equation}
with respect to some local basis $\{e^1,\dots, e^7\}$ of the $1$-forms on $M$, which we call an \emph{adapted basis}.
The notation $e^{i_1\dots i_k}$ stands for $e^{i_1}\wedge\dots\wedge e^{i_k}$. The fundamental 3-form $\varphi$ is stable
in the sense that its orbit at each point $p\in M$ under the natural action of the group $\textrm{GL}(T_pM)$ is open (see \cite{Hi}).

The existence of a $\mathrm{G}_2$ form $\varphi$ on a manifold $M$ induces a Riemannian metric $g_{\varphi}$ and a volume element $vol_\varphi$
on $M$  related by the formula:
\begin{equation}\label{metric}
g_{\varphi} (X,Y) vol_\varphi=\frac 16 \iota_X\varphi \wedge \iota_Y\varphi \wedge \varphi,
\end{equation}
for any vector fields $X, Y$ on $M$.

If $\nabla$ denotes the Levi-Civita connection with respect to the induced metric $g_\varphi$, Fern\'andez and Gray~\cite{FernandezGray} defined 
many different $\mathrm{G}_2$-structures in terms of the intrinsic torsion of the $\mathrm{G}_2$-structure given by $\nabla\varphi$. Moreover, it is proved that the intrinsic torsion is
completely determined by the exterior derivative
of the $\mathrm{G}_2$ form $\varphi$ and the 4-form $\ast\varphi$, where $\ast$ denotes the Hodge star operator induced by the metric and the volume form \eqref{metric}.
The most restrictive class of $\mathrm{G}_2$-structures is the one contaning the so called parallel  $\mathrm{G}_2$-structures, which are covariantly constant with respect to $\nabla$.  Manifolds endowed with such structure are characterized by the condition that both $\varphi$ and $\ast\varphi$ are
closed. Equivalently, the $\mathrm{G}_2$-form $\varphi$ and the Riemannian holonomy group of
the underlying metric $g_\varphi$ is a subgroup of $\mathrm{G}_2$ being in addition Ricci-flat~\cite{Br-0}. % like for example\footnote{Rq, mayo 2019: yo quitaria todos estos ejemplos}, parallel, nearly parallel, closed, coclosed, Locally Conformal Parallel, and so on.

The development of flows in Riemannian geometry has been mainly motivated by the study of the Ricci flow.
%The Ricci flow became a very important issue in Riemannian geometry and has been deeply studied.
The same techniques are also useful in the
study of flows involving other geometrical structures, like for example,
the K\"{a}hler Ricci flow. %that was studied by Cao in \cite{Cao}.

Given a closed (or calibrated in the terminology of Harvey and Lawson~\cite{HL}) $\mathrm{G}_2$-structure $\varphi_0$ on a manifold $M$, that is $d\varphi_0=0$,
Bryant introduced in
\cite{Br}
a natural flow, the so-called {\em Laplacian flow}, given by the initial value problem
$$
\left \{   \begin{array}{l} \frac{d}{dt} \varphi (t) = \Delta_t  \varphi(t),\\[3pt]
d\varphi(t)=0,\\
\varphi(0)=\varphi_0,
\end{array}
\right.
$$
where $\Delta_t$ is the Hodge
Laplacian operator of the metric $g_{\varphi(t)}$ determined by $\varphi(t)$. The short time existence and uniqueness of solution for the Laplacian flow of any closed $\mathrm{G}_2$-structure,
on a compact manifold~$M$, has been proved by Bryant and Xu in the unpublished paper \cite{Br-Xu}.
Also, long time existence and convergence of the Laplacian flow
starting near a torsion-free $\mathrm{G}_2$-structure was proved in  the unpublished paper \cite{Xu-Ye} whenever the torsion
of $\varphi$ is sufficiently small. In the last years, Lotay and Wei in \cite{LW1}, \cite{LW2} and \cite{LW3} have obtained many results concerning the properties of
the Laplacian flow.

In \cite{FFM} the first examples  of noncompact manifolds with long time existence of the solution for the
Laplacian flow of a closed $\mathrm{G}_2$-structure are shown. Those examples are nilpotent Lie groups admitting an invariant
closed $\mathrm{G}_2$-structure which determines the nilsoliton metric. Recently in \cite{FR} the authors studied the Laplacian flow of a closed $\mathrm{G}_2$-structure on warped products
of the form $M \times S^1$ where the base space is a 6-dimensional compact manifold endowed with an $\mathrm{SU}(3)$-structure. Impossing the warping function to be constant they find sufficient
conditions for the existence of solution of the Laplacian flow  and present some examples where $M$ is a six-dimensional solvmanifold.

Karigiannis, McKay and Tsui in \cite{KMT} introduced the \emph{Laplacian coflow} (or \emph{coflow} for short). In this case the initial $\mathrm{G}_2$-form $\varphi_0$ is claimed to
be coclosed (or cocalibrated as in~\cite{HL}),
i.e. $d\psi_0=0$,  where $\psi_0=\ast\varphi_0$.
The equations of this flow  are given by
$$
\left \{   \begin{array}{l} \frac{d}{dt} \psi (t) = - \Delta_t  \psi(t),\\[3pt]
d\psi(t)=0,\\
\psi(0)=\psi_0,
\end{array}
\right.
$$
with $\psi(t)=\ast_t \varphi(t)$ the Hodge dual 4-form of the $\mathrm{G}_2$-structure $\varphi(t)$ and $\Delta_t$ is the Hodge Laplacian
operator with respect to the metric $g_{\varphi(t)}$ induced by $\varphi(t)$. Unlike the Laplacian flow, up to now short time existence of solution of the coflow is not known.
Assuming short time existence and uniqueness of solution, it is shown in \cite{KMT} that the condition of the initial $\mathrm{G}_2$-form $\varphi_0$ to be coclosed (equiv. $\psi_0$ closed) is
preserved along the flow.

In \cite{G} Grigorian introduced a modified version of the Laplacian coflow which is called the \emph{modified Laplacian coflow} and proved short time existence and
uniqueness of solution for this modified flow. Recently in \cite{BFF} explicit solutions for the coflow and the modified Laplacian coflow have been described.
These solutions are one-parameter families of $\mathrm{G}_2$-structures defined on the 7-dimensional Heisenberg Lie group. The solutions for the coflow are always ancient, i.e., defined for all time $-\infty<t<T$, with $T<\infty$, for every initial cocalibrated $\mathrm{G}_2$-structure. The condition of the induced metric to be Ricci soliton is preserved along the coflow.
For overviews on these topics, see \cite{G2} and \cite{Lot}.

In this paper we are concerned with studying the Laplacian flow, resp. coflow, on Locally Conformal Parallel $\mathrm{G}_2$-structures, (LCP for short) as they play in some sense an intermediate role between closed and coclosed $\mathrm{G}_2$-structures.  LCP $\mathrm{G}_2$-structures are 
 characterized by the fact that at each point $p\in M$, there is some differentiable function $f$ defined on a local neighbourhood of $p$ such
that the underlying metric $g_\varphi$ can be modified locally to a metric $\widetilde g$ with holonomy a subgroup of $\mathrm{G}_2$ by means of
a conformal change $\widetilde g= e^{2f}g_\varphi$. Equivalently, the LCP condition  is given in terms of the exterior derivatives of $\varphi$ and $\ast\varphi$ by:
\begin{equation}\label{sigma-LCP}
\begin{aligned}
d\varphi &= 3 \, \tau \wedge \varphi, \quad d\ast \varphi &= 4 \, \tau \wedge \ast \, \varphi,
\end{aligned}
\end{equation}
with $\tau$ the Lee 1-form. These $\mathrm{G}_2$-structures are of type $\mathcal{X}_4$ in the sense of Fern\'andez-Gray, see \cite{FernandezGray}.

%\begin{center}
%    \noindent\begin{minipage}{0.4\textwidth}
%\begin{equation*}
%\left\{
%\begin{aligned}
%&\frac{d}{dt} \varphi(t)=\Delta_{t} \varphi(t), \\
%&\varphi(t)\in \mathcal X_4.
%\end{aligned}
%\right.
%\end{equation*}
%    \end{minipage}%
%    \begin{minipage}{0.1\textwidth}\centering
% \
%    \end{minipage}%
%    \begin{minipage}{0.4\textwidth}
%\begin{equation*}
%\left\{
%\begin{aligned}
%&\frac{d}{dt} \psi(t)=-\Delta_{t} \psi(t), \\
%&\ast_t\psi(t)\in \mathcal X_4.\\
%\end{aligned}
%\right.
%\end{equation*}
%    \end{minipage}\vskip1em
%\end{center}

 In order to describe the first examples of solution of these flows we will consider the class of solvable Lie groups  described by Fino and Chiossi
 in~\cite{CF} constructed as rank-one solvable
extensions of nilpotent Lie groups admitting left-invariant Locally Conformal Parallel $\mathrm{G}_2$-structures.

%\begin{align*}
%\mathfrak{cp}^m_1 = &(-me^{17}, -m e^{27}, -me^{37}, -m e^{47}, -me^{57}, -m e^{67}, 0 );\\
%\mathfrak{cp}^m_2 = &\Big(-\frac43m e^{17}+\frac23me^{36}, -m e^{27}, -\frac23me^{37}, -m e^{47}, -me^{57}, -\frac23 m e^{67}, 0 	\Big);\\
%\mathfrak{cp}^m_3 = &\Big(-\frac32m e^{17}+\frac12 m e^{36}+\frac12 m e^{45}, -m e^{27}, -\frac34me^{37}, -\frac34 m e^{47}, -\frac34me^{57}, -\frac34m e^{67}, 0 	\Big);\\
%\mathfrak{cp}^m_4 = &\Big(-\frac75m e^{17}+\frac25 m e^{36}+\frac25 m e^{45}, -\frac65m e^{27}-\frac25m e^{46}, -\frac45me^{37},  %\\
%%& \hspace{7.5cm}
%-\frac35 m e^{47},  -\frac45me^{57},-\frac35m e^{67}, 0 \Big);\\
%\mathfrak{cp}^m_5 = &\Big(-\frac54 m e^{17}+\frac12 m e^{45}, -\frac54m e^{27}-\frac12m\, e^{45}, -me^{37}, -\frac12 m e^{47}, %\\
%%& \hspace{9cm}
% -\frac34me^{57}, -\frac34m e^{67}, 0 	\Big);\\
%\mathfrak{cp}^m_6 = &\Big(-\frac43m e^{17}+\frac13 m e^{36}+\frac13 m e^{45}, -\frac43m e^{27}+\frac13m e^{35}-\frac13 m e^{46}, -\frac23me^{37}, %\\
%%& \hspace{7.5cm}
% -\frac23 m e^{47}, -\frac23me^{57}, -\frac23m e^{67}, 0 	\Big);\\
%\mathfrak{cp}^m_7 = &\Big(-\frac65m e^{17}+\frac25 m e^{36}, -\frac35m e^{27}, -\frac35me^{37}, \frac25 m e^{26}-\frac65 m e^{47}, %\\
%%& \hspace{7.5cm}
%\frac25 m e^{23}-\frac65me^{57}, -\frac35m e^{67}, 0 	\Big);
%\end{align*}
%where $\{e^1,\dots,e^7\}$ is a basis of invariant 1-forms of every solvmanifold and $m\in\mathbb{R}^*$.  It turns out that
%the 3-form $\varphi_0$ given by~\eqref{sigma} in terms of the basis $\{e^1,\dots,e^7\}$ defines an invariant LCP structure on every solvmanifold.

The paper is structured as follows: in Section 1, we review some explicit examples on Lie groups solving the Laplacian flow and the
Laplacian coflow. This allows us to set a generic
ansatz for solving flows related with $\mathrm{G}_2$-structures  on Lie groups which will be useful in the rest of the paper.
Section 2 starts by introducing rank-one solvable extensions of nilpotent Lie groups.  We detailed in Proposition~\ref{algebrasFR} the list of Lie algebras
found in~\cite[Theorem 1]{CF} underlying the seven dimensional solvable Lie groups constructed in this way and admitting a left-invariant
LCP $\mathrm{G}_2$-structure. The rest of this section deals with exploring the Laplacian flow under the assumption of solutions defined in Section 1
either setting necessary and sufficient conditions preserving the LCP condition or describing the $\Delta_t\varphi(t)$ in a suitable form. 
Sections 3 and 4 are devoted to construct explicit examples of solutions to the Laplacian flow and coflow preserving the LCP-condition.  In Theorem 3.1 we present an explicit solution for the Laplacian flow 
where the LCP-condition is preserved, notice that the metric induced by the solution remains Einstein along the flow. The rest of Section 3 is devoted to obtain solutions for the remaining solvable Lie groups described by Chiossi and Fino. The solutions of the flow turn out to be Laplacian solitons.   In Section 4 we obtain relations between the sets of solutions of the Laplacian flow and coflow
where the LCP-condition is preserved (see Theorem~\ref{Thm:flujo-coflujo}).  Finally, in the Appendix we include the expressions of the curvature for the metric induced by the solutions of the Laplacian flow previously obtained.

\end{section}

%%%%%%%%%%%%%%%%%%%%%%%%%%%%%%%%%%%%%%%%%%%%%%%%%%%%%%%%%%%%%%%%%%%%%%%%%%%%%%%%%%%%%%
%4. LAPLACIAN FLOW%%%%%%%%%%%%%%%%%%%%%%%%%%%%%%%%%%%%%%%%%%%%%%%%%%%%%%%%%%%%%%%%%%%%
%%%%%%%%%%%%%%%%%%%%%%%%%%%%%%%%%%%%%%%%%%%%%%%%%%%%%%%%%%%%%%%%%%%%%%%%%%%%%%%%%%%%%%

\begin{section}{Laplacian flows on Lie groups}\label{sectLaplacian}

In the last years there has been a wide interest in finding solutions for the Laplacian flow and related notions have been explored such as new examples with extra properties.
In general, flows of $\Gtwo$-structures are of the form
\begin{equation}\label{Gen-flow}
\left \{   \begin{array}{l} \frac{d}{dt} \varphi (t) = \Delta_t  \varphi(t),\\[3pt]
\varphi(t)\in \mathcal C,
\end{array}
\right.
\end{equation}
where $\Delta_t$ denotes the corresponding Hodge Laplacian operator, $\mathcal C$ is a specific class of $\Gtwo$-structures and $t$ lives in an open real interval.  We will refer to it as \emph{$\mathcal C$-flow}.

The first author considering flows of $\mathrm{G}_2$-structures was Bryant in \cite{Br}. The objective of considering flows of $\mathrm{G}_2$-structures was to obtain examples of $\mathrm{G}_2$-manifolds without torsion as the result of certain evolution of other $\mathrm{G}_2$-structures with torsion. Thus, Bryant considered the so-called \emph{Laplacian flow} of a $\mathrm{G}_2$-structure $\varphi_0$ which is given by \eqref{Gen-flow} where $\varphi(t)$ is supposed to be closed.
%\begin{equation}\label{Lapflow}
%\left \{   \begin{array}{l} \frac{d}{dt} \varphi (t) = \Delta_t  \varphi(t),\\[3pt]
%\varphi(0)=\varphi_0,\\
%d\varphi(t)=0,
%\end{array}
%\right.
%\end{equation}
 On compact manifolds short time existence and uniqueness of solution for the Laplacian flow of a
closed $\mathrm{G}_2$-structure has been proved by Bryant and Xu in \cite{Br-Xu}. Xu and Ye in \cite{Xu-Ye} proved long time existence and convergence
of solution of the Laplacian flow starting near a torsion-free $\mathrm{G}_2$-structure. In the last years Lotay and Wei in the series of papers \cite{LW3, LW2, LW1} have obtained important results concerning long time existence and convergence of solutions of the Laplacian flow.      

On the other hand, in \cite{KMT} Karigiannis, McKay and Tsui introduced the \emph{Laplacian coflow}. This latter flow can be considered as the analogue to the Laplacian flow
in which the fundamental 3-form is claimed to be coclosed instead of closed. Thus, this flow is given by the equations
\begin{equation}\label{co-F}
\left \{   \begin{array}{l} \frac{d}{dt} \psi (t) = - \Delta_t  \psi(t),\\[3pt]
\psi(0)=\psi_0,\\
d\psi(t)=0,
\end{array}
\right.
\end{equation}
with $\psi(t)=\ast_t\varphi(t)$ and $\ast_t$ denoting the Hodge star operator. As far as the authors know, short time existence and uniqueness of solution for this latter flow is not known. In \cite{G} Grigorian introduced a modified version of this flow called modified Laplacian coflow for which he proved short time existence and uniqueness of solution.

\subsection{Torsion of $\mathrm{G}_2$-structures} The torsion of a $\mathrm{G}_2$-structure can be identified with the covariant derivative of the fundamental form $\varphi$ with respect to the Levi-Civita connection of the induced metric.   As it is described in  \cite{FernandezGray}, it can be decomposed into four $\mathrm{G}_2$ irreducible components, namely $X_1, X_2, X_3$ and $X_4$.
Thus, a $\mathrm{G}_2$-structure is said to be of type $\mathcal{P}, \mathcal{X}_i, \mathcal{X}_i \oplus \mathcal{X}_j, \mathcal{X}_i \oplus \mathcal{X}_j \oplus \mathcal{X}_k $ or $\mathcal{X}$ if  $\nabla \varphi$ lies in $\{0\}, X_i, X_i \oplus X_j, X_i \oplus X_j \oplus X_k $ or $X=X_1 \oplus X_2 \oplus X_3 \oplus X_4$, respectively.
Hence, there exist 16 different classes of $\mathrm{G}_2$-structures. Some of the principal classes are summarized in Table~\ref{fig:1} and Figure \ref{fig:1}.

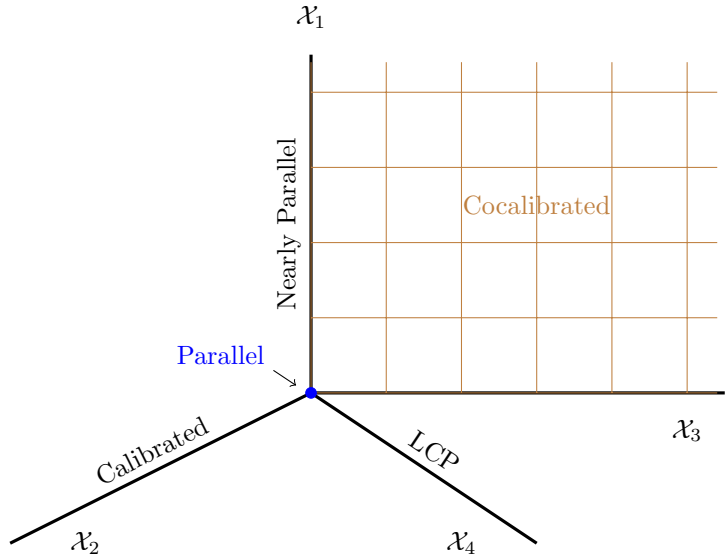
\begin{figure}[h]
\begin{center}
\begin{tikzpicture}[scale=1.0]
%\draw[dashed, very thick] (0.1,-1) .. controls (0,2) and (10,3) .. (10,1);
\draw [very thick][-] (0,0) -- (0.0cm,4.5cm) node[midway, sloped, above] {Nearly Parallel};
\draw [very thick](0,5) node {$\mathcal{X}_1$}; %Nombre del primer vector
\draw [very thick][-] (0,0) -- (5.5cm,0.0cm) ;
\draw [very thick](5,-0.5) node {$\mathcal{X}_3$}; %Nombre del primer vector
\draw [very thick][-] (0,0) -- (-4.0cm,-2.0cm)node[midway, sloped, above] {Calibrated} ;
\draw [very thick](-3.0,-2.0) node {$\mathcal{X}_2$}; %Nombre del primer vector
\draw [very thick][-] (0,0) -- (3.0cm,-2.0cm)node[midway, sloped, above] {LCP} ;
\draw [very thick](2.0,-2.0) node {$\mathcal{X}_4$}; %Nombre del primer vector
\draw [brown, very thick](3.0,2.5) node {Cocalibrated}; %Nombre del primer vector
\draw [->]  (-0.5cm,0.3cm)--(-0.2,0.1) ;
 \draw[brown, step=1.0cm] (0,0) grid (5.4,4.4);
%\draw [->] (2.9,2.2) -- (4.0cm,2.7cm) ;
%\draw [->] (3.8,2.4) -- (4.8cm,2.8cm) ;
\filldraw [blue] (0,0) circle (2pt);
\draw [blue, very thick](-1.2,0.5) node {Parallel}; %Nombre del primer vector
\end{tikzpicture}
\end{center}
\caption{Principal classes of $\mathrm{G}_2$-structures} \label{fig:1}
\end{figure}
Equivalently these classes of $\mathrm{G}_2$-structures can be characterized in terms of the exterior derivatives of $\varphi$ and $\ast\varphi$~\cite{FernandezGray}.

\begin{table}[ht]
\centering
\begin{tabular}{ccl}
\hline
\textbf{Class} & \textbf{Type} & \textbf{Exterior derivatives}\\  \hline
${\mathcal P}$          			          	& parallel 			                    & $d\varphi=0,\qquad\qquad d\ast\varphi=0$  	\\
${\mathcal X}_2$               			& calibrated (or closed)	  	      	    & $d\varphi=0$	\\
${\mathcal X}_4$               			& locally conformal parallel (LCP) & $d\varphi = 3 \, \tau \wedge \varphi, \quad d\ast \varphi = 4 \, \tau \wedge \ast \, \varphi$  	\\
${\mathcal X}_1 \oplus {\mathcal X}_3$     & cocalibrated   (or coclosed)     & $d\ast\varphi=0$ 		\\  \hline
\end{tabular}
\caption{Some classes of $\mathrm{G}_2$-structures.}\label{G2classes}
%\vspace{0.1cm}
\end{table}

%In general Laplacian flows of $\mathrm{G}_2$-structures consist in a system of the form
%$$
%\left \{   \begin{array}{l} \frac{d}{dt} \varphi (t) = \Delta_t  \varphi(t),\\[3pt]
%\varphi(t) \in  \mathcal{X},\\
%\varphi(0)=\varphi_0,
%\end{array}
%\right.
%$$
%where $\Delta_t$ denotes the Hodge Laplacian operator and $\mathcal{X}$ one of the Fern\'andez-Gray classes of $\mathrm{G}_2$-structures.
%

\begin{subsection}{Examples of solution}
The first examples of long time solutions for the Laplacian flow of closed $\mathrm{G}_2$-structures ($\mathcal{C}=\mathcal X_2$) were described in \cite{FFM} using nilpotent Lie groups endowed with a one parameter family of left-invariant closed $\mathrm{G}_2$-structures.

\begin{example} Consider the connected and simply connected Lie group $G$ whose underlying Lie algebra has the structure equations:
\begin{equation*}
de^{5}= e^1 \wedge e^2, \quad de^6= e^1 \wedge e^3, \text{ and } de^i=0 \text{ for all } i=1,2,3,4,7.
\end{equation*}
The family of closed $\mathrm{G}_2$ forms $\varphi(t)$ on $G$ given by
\begin{equation*}\label{solution:N2}
\varphi(t)=e^{147}+e^{267}+e^{357}+f(t)^{3}e^{123}+e^{156}+e^{245}-e^{346}, \qquad t\in \left (-\frac{3}{10},+ \infty \right),
\end{equation*}
where $f(t)$ is the positive function
\begin{equation*}
f(t)=\Big(\frac{10}{3} t +1\Big)^{\frac{1}{5}}.
\end{equation*}
is the solution of the Laplacian flow with initial value
\begin{equation*}
\varphi_0=e^{147}+e^{267}+e^{357}+e^{123}+e^{156}+e^{245}-e^{346}.
\end{equation*}
\end{example}

%In order to obtain this solution, the authors consider $x^i = x^i(t) = f_i(t)e^{i}$ where $f_1(t) = f_2(t) = f_3(t) =f(t)$ and $f_4(t) = f_5(t) = f_6(t) = f_7(t) = f(t)^{-1/10}$.\footnote{Rq, 27 mayo 2019:  esto debe ser $f^{-1/2}$}

%\footnote{Rq, 27 mayo 2019.  Aqui habria que poner algun ejemplo de Axioms.}

%Solutions to the Laplacian flow where the $\mathrm{G}_2$-structure is requiered to be Locally Conformally Closed $(\varphi \in \mathcal{X}_1 \oplus \mathcal{X}_2)$ have been obtained in \cite{FMS}.
%\begin{example}
%Consider $G$ the simply connected and solvable Lie group of dimension $7$
%whose Lie algebra $\mathfrak{g}$ is defined by
%\begin{align*}
%de^{1}= &e^3 \wedge e^7, \quad de^2= e^4 \wedge e^7, \quad de^{3}= -e^1 \wedge e^7, \quad de^4= -e^2 \wedge e^7,  \\
%de^{5}= &e^1 \wedge e^4 + e^2 \wedge e^3, \quad de^6= e^1 \wedge e^3 - e^2 \wedge e^4, \text{ and } de^7=0.
%\end{align*}
%The family of closed $G_2$ forms $\varphi(t)$ on $G$ given by
%\begin{equation}\label{eq:solution1}
%\varphi(t)=e^{127}+e^{347}+\left(1-\tfrac83t\right)^{-3/2} \left(e^{567}+e^{135}-e^{146}-e^{236}-e^{245}\right)
%\end{equation}
%is the solution of the Laplacian flow with initial value
%\begin{equation*}
%\varphi_0=e^{127}+e^{347}+e^{567}+e^{135}-e^{146}-e^{236}-e^{245}.
%\end{equation*}
%\end{example}

\medskip

 Analogously in \cite{BFF} have been given explicit long time solutions for the Laplacian coflow~\eqref{co-F}.
 \begin{example}
Consider the 7-dimensional Heisenberg Lie group $H_7$, whose
%
% which is given by the matrices of the form
%
%\begin{equation*}
%a=\left( {\begin{array}{ccccc}
%  1 & x_1 & x_3 & x_5 & x_7\\
% & 1 &  &  & x_2\\
%  &  & 1 &  & x_4\\
%   &  &  & 1 & x_6\\
%    &  &  &  & 1\\
% \end{array} } \right)
%\end{equation*}
%with $x_i \in \mathbb{R}$ for all $i=1,\dots 7$. Then a global system of coordinates ${x_i}$ for $H_7$ is defined by $x_i(a) = x_i$. A standard calculation shows that a basis
%for the left invariant 1-forms on $H_7$ can be described by
%\begin{align*}
%e^1 &= dx_1, \quad e^2 = dx_2, \quad  e^3 = dx_3, \quad e^4 = dx_4, \\
%e^5 &= dx_5,  \quad e^6 =dx_6,  \quad \text{ and }   \quad e^7 = dx_7 -x_1 dx_2 -x_3 dx_4 -x_5 dx_6.
%\end{align*}
corresponding Lie algebra, namely $\mathfrak{h}_7$, is given by the structure equations
\begin{equation*}
de^{7}= \frac{\sqrt6}{6}(e^1 \wedge e^2 + e^3 \wedge e^4 + e^5 \wedge e^6), \text{ and } de^i=0 \text{ for all } i=1,\dots,6.
\end{equation*}
The solution of the Laplacian coflow on $H_7$ with the initial coclosed $\mathrm{G}_2$ form,
\begin{equation*}
\varphi_0=e^{127}+e^{347}+e^{567}+e^{135}-e^{146}-e^{236}-e^{245},
\end{equation*}
is given by
\begin{equation*}
\varphi(t) = \frac{1}{f(t)} (e^{127}+e^{347}+e^{567})+f(t)^3(e^{135}-e^{146}-e^{236}-e^{245}), \quad t \in \Big(-\infty, \frac35 \Big)
\end{equation*}
where $f(t)$
 is the positive function
 \begin{equation*}
f(t) =\Big( 1- \frac53 t\Big)^{\frac{1}{10}}.
 \end{equation*}
\end{example}
%In this case, the authors have considered $x^i = x^i(t) = f_i(t)e^{i}$ where $f_i(t) =f(t)$ for $i=1,\ldots, 6$ and $f_7(t)= f(t)^{-3}$.

\bigskip

\end{subsection}

\begin{subsection}{Results on Lie groups}

%Let $S_s$ be a solvmanifold with underlying Lie algebra $\mathfrak{cp}^m_i$. In what follows we will consider a basis $\{x^1,\dots,x^7\}$ of invariant 1-forms on $S_s$ given by:
%$$
%x^i \equiv x^i(t) = f_i(t) e^i,
%$$
%where $f_i(t)$ are differentiable functions,  $f_i(t)\neq 0$ in an open real interval and $f_i(0) = 1,$ for $i=1,\ldots, 7$.  Using this, we consider the one-parameter family of $\mathrm{G}_2$-structures on the solvmanifold given by:

Notice that the previous examples consist on solutions of the flows on Lie groups where a very concrete ansatz has been considered. In general, let $G$ be a simply connected solvable Lie group of dimension $7$ with Lie algebra $\mathfrak{g}$.
Let $\{e^1, \dots, e^7\}$ be a basis of the dual space ${\frg}^\ast$ of ${\frg}$, and
let $f_i=f_i(t)$ $(i=1,\dots,7)$ be some differentiable real functions depending on a parameter
$t\in I\subset{\mathbb{R}}$ such that $f_i(0)=1$ and $f_i(t)\neq0$, for any $t\in I$, where $I$ is a real open interval.
For each $t\in I$, we define the basis $\{x^1,\dots,x^7\}$ of ${\frg}^\ast$ by
$$
x^i=x^i(t)=f_{i}(t)e^i, \quad 1\leq i\leq 7.
$$
We consider the one-parameter family of left-invariant $\Gtwo$-structures $\varphi(t)$ on $G$ given by
\begin{align}
\begin{split}\label{eqn:expression_sol}
\varphi(t)&=x^{127}+x^{347}+x^{567}+x^{135}-x^{146}-x^{236}-x^{245}\\
&=f_{127}e^{127}+f_{347}e^{347}+f_{567}e^{567}+f_{135}e^{135}-f_{146}e^{146}-f_{236}e^{236}-f_{245}e^{245},
\end{split}
\end{align}
where $f_{ijk}=f_{ijk}(t)$ stands for the product  $f_{i}(t)f_{j}(t)f_{k}(t)$.
Now, following \cite{FMS} can be introduced the function $\varepsilon(i,j,k)$ on ordered indices $(i,j,k)$ as follows:
\[
\varepsilon(i,j,k)=\begin{cases}
1&\text{ if }(i,j,k)\in  A=\{(1,2,7),(1,3,5),(3,4,7),(5,6,7)\};\\
-1&\text{ if }(i,j,k)\in  B=\{(1,4,6), (2,3,6),(2,4,5)\};
%0&\text{ otherwise.}\\
\end{cases}
\]
Thus, the $\Gtwo$ form $\varphi$ defined in~\eqref{sigma}, can be expressed as
$\varphi=\sum_{(i,j,k)\in A\cup B} \varepsilon(i,j,k)e^{ijk}$, and
the family of $\Gtwo$ forms $\varphi(t)$ given by \eqref{eqn:expression_sol} becomes
$$
\varphi(t)=\sum_{(i,j,k)\in A\cup B} \varepsilon(i,j,k)x^{ijk}.
$$
Therefore,
\begin{align*}
\frac{d}{dt}\varphi(t)&=\sum_{(i,j,k)\in A\cup B}\varepsilon(i,j,k) {d f_{ijk}\over dt}e^{ijk}=\sum_{(i,j,k)\in A\cup B} \varepsilon(i,j,k)\frac{(f_{ijk})'}{f_{ijk}}x^{ijk}.%\label{eq:coeffdphi}\\
%&=\sum_{(i,j,k)\in A\cup B} \varepsilon(i,j,k){d\over dt}\left(\ln f_{ijk}\right)x^{ijk}.%\label{eq:coeffdphiln}
\end{align*}
Moreover, we express the 3-form $\Delta_t\varphi(t)$ as a linear combination of the basis of 3-forms $\{x^{abc}\}$ as
\begin{equation}\label{eq:delta3ijk}
\Delta_t \varphi(t)=\sum_{(i,j,k)\in A\cup B} \varepsilon(i,j,k)\Delta_{ijk}\,x^{ijk} + \sum_{1\leq i<j<k\leq 7, \, (i,j,k)\not\in A\cup B} \Delta_{ijk}\,x^{ijk}\,,
\end{equation}
where $\varepsilon(i,j,k)\Delta_{ijk}$ is the coefficient in $x^{ijk}$ of $\Delta_t\varphi(t)$ if $(i,j,k)\in A\cup B$ (i.e., if $\varepsilon(i,j,k)\neq0$), and
$\Delta_{ijk}$ is the coefficient in $x^{ijk}$ of $\Delta_t\varphi(t)$ if $(i,j,k)\not\in A\cup B$.
Consequently, the first equation of the $\mathcal C$-flow \eqref{Gen-flow} (regardless of condition $\mathcal C$) is equivalent to the system of differential equations
\begin{gather}\label{eq:evol_eq_expand1}
\begin{cases}
\Delta_{ijk}=\frac{(f_{ijk})'}{f_{ijk}} \qquad\,\quad \,\, \text{ if }\,  (i,j,k)\in A\cup B,\\
\Delta_{ijk}=0 \qquad\,\,\,\qquad  \,\, \text{ if }\,  1\leq i<j<k\leq 7\, \text{ and }\, (i,j,k)\not\in A\cup B.
\end{cases}
\end{gather}
%that is,
%\begin{gather}\label{eq:evol_eq_expand2}
%\begin{cases}
%\Delta_{ijk}={d\over dt}\ln(f_{ijk}) \quad\, \,\,\,\text{ if }\,  (i,j,k)\in A\cup B,\\
%\Delta_{lmn}=0 \qquad\,\,\,\qquad  \,\, \text{ if }\,  1\leq l<m<n\leq 7\, \text{ and }\, (l,m,n)\not\in A\cup B.
%\end{cases}
%\end{gather}
%
The following lemma generalizes \cite[Lemma 1]{FMS} and states some properties involving the $\Delta_{ijk}$ coefficients.

%\begin{lemma}\label{lem:equalities}
%Let $\varphi(t)$ be a family of left invariant $\Gtwo$-structures on the Lie group $G$ solving the system \eqref{eq:evol_eq_expand1}, and such that $\varphi(t)$ can be expressed as \eqref{eqn:expression_sol},
%for some functions $f_{i}=f_{i}(t)$.
% For ordered indices $(i,j,k)$ and $(p,q,r)\in A\cup B$ (that is, $\varepsilon(i,j,k)$ and $\varepsilon(p,q,r)$ are both non-zero) we have
%\begin{enumerate}[label=\roman*)]
%\item if $\Delta_{ijk}=\Delta_{pqr}$, then $f_{ijk}=f_{pqr}$;\label{item:2}\\
%\item if $f_{ijk}\Delta_{ijk}=f_{pqr}\Delta_{pqr}$, then $f_{ijk}=f_{pqr}$;\label{item:2}\\
%\item if $\Delta_{ijk}+\Delta_{pqr}=0$, then $f_{ijk}f_{pqr}=1$;\label{item:3}\\
%\item if $f_{ijk}\Delta_{ijk}+f_{pqr}\Delta_{pqr}=0$, then $f_{ijk}+f_{pqr}=2$.\label{item:4} \\
%\end{enumerate}
%\end{lemma}

\begin{lemma}\label{lem:equalities2}
Let $\varphi(t)$ be a family of left invariant $\Gtwo$-structures given by \eqref{eqn:expression_sol} on the Lie group $G$ solving the system \eqref{eq:evol_eq_expand1}.
 For ordered indices $(i,j,k)$ and $(p,q,r)\in A\cup B$ and  $\alpha, \beta\in \mathbb R$ we have
\begin{enumerate}[label=\roman*)]
\item if $\alpha\Delta_{ijk}=\beta\Delta_{pqr}$, then $(f_{ijk})^{\alpha}=(f_{pqr})^{\beta}$;\label{item:1}\\[-4pt]
\item if $\alpha f_{ijk}\Delta_{ijk}=\beta f_{pqr}\Delta_{pqr}$, then $\alpha(f_{ijk}-1)=\beta(f_{pqr}-1)$.\label{item:2}
\end{enumerate}
\begin{proof}
For $i)$ suppose $\alpha \Delta_{ijk}=\beta \Delta_{pqr}$.  In view of \eqref{eq:evol_eq_expand1}  this is equivalent to $\alpha \frac{(f_{ijk})'}{f_{ijk}}=\beta \frac{(f_{pqr})'}{f_{pqr}}$. Notice that the last expression can be stated as
$\alpha \frac{d}{dt}\ln(f_{ijk})=\beta \frac{d}{dt}\ln(f_{pqr})$. Therefore $\frac{d}{dt}\ln\big(\frac{(f_{ijk})^{\alpha}}{(f_{pqr})^{\beta}}\big)=0$. Hence  $\ln\big(\frac{(f_{ijk})^{\alpha}}{(f_{pqr})^{\beta}}\big)$ is constant and since $f_{l}(0)=1$ for all $l=1, \dots, 7$ we conclude that $(f_{ijk})^{\alpha}=(f_{pqr})^{\beta}$.  Part $ii)$ is immediate.
\end{proof}
\end{lemma}

Notice that flows on $\mathrm{G}_2$-structures whose fundamental form is claimed to be calibrated (belonging to class $\mathcal{C}=\mathcal{X}_2$) or
cocalibrated (in class $\mathcal{C}=\mathcal{X}_1 \oplus \mathcal{X}_3$) have been deeply studied. Thus in view of diagram 1 it seems natural to consider the remaining case, i.e. flows
where the fundamental form is required to be Locally Conformal Parallel ($\mathcal{C}=\mathcal{X}_4$). However, as far as the authors know, nothing has been done for flows of $\mathrm{G}_2$-structures where the LCP condition is required along the flow. Therefore in this paper we are concerned with studying the Laplacian flow, resp. coflow, of an LCP $\mathrm{G}_2$-structure on a manifold $M$, or simply the \emph{LCP-flow}, resp. \emph{LCP-coflow}, which can be defined as:

\begin{center}
    \noindent\begin{minipage}{0.4\textwidth}
\begin{equation}\label{LCP-flow-eq}
\begin{cases}
\dfrac{d}{dt} \varphi(t)=\Delta_{t} \varphi(t), \\[5pt]
\varphi(t)\in \mathcal X_4
\end{cases}
\end{equation}
    \end{minipage}%
    \begin{minipage}{0.1\textwidth}\centering
 \
    \end{minipage}%
    \begin{minipage}{0.4\textwidth}
\begin{equation}\label{LCP-coflow-eq}
\begin{cases}
\dfrac{d}{dt} \psi(t)=-\Delta_{t} \psi(t), \\[5pt]
\ast_t\psi(t)\in \mathcal X_4
\end{cases}
\end{equation}
    \end{minipage}\vskip1em
\end{center}
where a $\mathrm G_2$-structure $\varphi$ belongs to class $\mathcal X_4$ if it satisfies equation~\eqref{sigma-LCP}.
%$$d \varphi = 3 \, \tau \wedge \varphi,\quad d \ast \varphi = 4 \, \tau \wedge \ast\varphi.$$
\end{subsection}
\end{section}

\section{Laplacian flow on  LCP rank-one solvable extensions of nilpotent Lie groups}
%\section{{\color{blue}Solution of the Laplacian flow and coflow of an LCP $\mathrm{G}_2$-structure on $\mathfrak{cp}_1^m$}}
In this section we study the Laplacian flow on a specific set of Lie groups endowed with a left-invariant LCP $\mathrm{G}_2$-structure.  The associated Lie algebras of these groups are rank-one solvable extensions of 6-dimensional nilpotent Lie algebras.
These solvable extensions are constructed generically as follows (see \cite{Will}). Given a $n$-dimensional metric nilpotent Lie algebra
$(\frn,\langle\cdot,\cdot\rangle_\frn)$, its $(n+1)$-dimensional solvable extension is a
vector space $\frs=\frn\oplus\R e_{n+1}$ where $e_{n+1}\notin\frn$ endowed with a metric $\langle\cdot,\cdot\rangle_\frs$ which is fixed on $\frs$ extending the one on $\frn$ i.e $\langle \cdot, \cdot \rangle_{\frs_{|_\frn}}=\langle \cdot, \cdot \rangle_\frn$ and declaring that $\langle e_{n+1}, e_{n+1}\rangle_\frs=1$ and
$\langle e_{n+1}, \frn\rangle_\frs=0$. Now, given a derivation $D$ of the Lie algebra $\frn$, the Lie bracket $[\cdot,\cdot]_\frs$ on $\frs$ is
defined as $[X,Y]_\frs=[X,Y]_\frn$ and $[e_{n+1},Y]_\frs=D Y$ for every $X,\,Y\in\frn$, i.e., $\text{ad}_{e_{n+1}}|_\frn=D$.

Fino and Chiossi~\cite{CF} adapt the former construction when $\frn$ is a six-dimensional nilpotent Lie algebra endowed with an SU(3)-structure $(\omega,\psi_+)$
and $D$ is a derivation of $\frak n$ being non-singular, self-adjoint with respect to $\langle \cdot,\cdot \rangle_\frs$ and diagonalisable by an adapted
Hermitian basis $\{e_1,\ldots,e_6\}$ of $\frn$ (the latter condition being equivalent to $(DJ)^2=(JD)^2$). In this setting, the Maurer-Cartan equations for the seven-dimensional Lie algebra yield
\begin{equation*}
\left\{\begin{array}{lcl}
de^k&=&\eta_k\, e^k\wedge e^7+\hat de^k,\quad  1\leq k\leq 6,\\[4pt]
de^7&=&0,
\end{array}\right.
\end{equation*}
where the $\eta_k$ are the eigenvalues of the derivation $D$ and $\hat de^k=\sum_{1\leq i<j\leq 6}c_{ij}^ke^{ij}$ is
the exterior derivative at the 6-dimensional level $\frn$.

It turns out that there is a natural $\mathrm{G}_2$-structure on the
solvable Lie group $S=N\times\R$ corresponding to the $3$-form:
$$
\varphi=\omega\wedge e^7+\psi_+
$$
where $e^7$ denotes the 1-form $\langle e_7,\cdot\rangle_\frs$ and $N$ is the nilpotent Lie group associated to the nilpotent Lie algebra $\frak n$.  Indeed, they prove that when $(S,\varphi)$ is locally conformal parallel the SU(3)-structure $(\omega,\psi_+)$
is half-flat, that is, $d\omega^2=0$ and $d\psi_+=0$. More concretely, the list of Lie algebras underlying such locally conformal parallel structures in
this setting is contained in the following classifying result (in  a slightly different representation with respect to the original one found
inside the proof of \cite[Theorem 1]{CF}):
\begin{proposition}\label{algebrasFR}
Let $N$ be a nilpotent Lie group of dimension 6 endowed with an invariant SU(3)-structure $(\omega,\psi_+)$. Suppose that there is a
non-singular and self-adjoint derivation $D$ of the Lie algebra
$\frn$ such that $(DJ)^2=(JD)^2$. Then, on the solvable extension $\frs=\frn\oplus\R e_7$ with $\text{ad}_{e_7}=D$, the $\mathrm{G}_2$-structure
$$
\varphi=(e^{12}+e^{34}+e^{56})\wedge e^7+e^{135}-e^{146}-e^{236}-e^{245},
$$
is locally conformal parallel if and only if $\frs$ is isomorphic to one of the following list:
\begin{align*}
\mathfrak{cp}^m_1 = &(-me^{17}, -m e^{27}, -me^{37}, -m e^{47}, -me^{57}, -m e^{67}, 0 );\\
\mathfrak{cp}^m_2 = &\Big(-\frac43m e^{17}+\frac23me^{36}, -m e^{27}, -\frac23me^{37}, -m e^{47}, -me^{57}, -\frac23 m e^{67}, 0 	\Big);\\
\mathfrak{cp}^m_3 = &\Big(-\frac32m e^{17}+\frac12 m e^{36}+\frac12 m e^{45}, -m e^{27}, -\frac34me^{37}, -\frac34 m e^{47}, -\frac34me^{57}, -\frac34m e^{67}, 0 	\Big);\\
\mathfrak{cp}^m_4 = &\Big(-\frac75m e^{17}+\frac25 m e^{36}+\frac25 m e^{45}, -\frac65m e^{27}-\frac25m e^{46}, -\frac45me^{37},  %\\
%& \hspace{7.5cm}
-\frac35 m e^{47},  -\frac45me^{57},-\frac35m e^{67}, 0 \Big);\\
\mathfrak{cp}^m_5 = &\Big(-\frac54 m e^{17}+\frac12 m e^{45}, -\frac54m e^{27}-\frac12m\, e^{46}, -me^{37}, -\frac12 m e^{47}, %\\
%& \hspace{9cm}
 -\frac34me^{57}, -\frac34m e^{67}, 0 	\Big);\\
\mathfrak{cp}^m_6 = &\Big(-\frac43m e^{17}+\frac13 m e^{36}+\frac13 m e^{45}, -\frac43m e^{27}+\frac13m e^{35}-\frac13 m e^{46}, -\frac23me^{37}, %\\
%& \hspace{7.5cm}
 -\frac23 m e^{47}, -\frac23me^{57}, -\frac23m e^{67}, 0 	\Big);\\
\mathfrak{cp}^m_7 = &\Big(-\frac65m e^{17}+\frac25 m e^{36}, -\frac35m e^{27}, -\frac35me^{37}, \frac25 m e^{26}-\frac65 m e^{47}, %\\
%& \hspace{7.5cm}
\frac25 m e^{23}-\frac65me^{57}, -\frac35m e^{67}, 0 	\Big).
\end{align*}
% $$\mathfrak{cp}_1^m=(-m\,e^{17},\, -m\,e^{27},\, -m\,e^{37},\, -m\,e^{47},\, -m\,e^{57},\, -m\,e^{67},\, 0)$$
%
%
%%-----------------  EQ 10
%
%
%$$\mathfrak{cp}_2^m=\left(-\frac43 m\,e^{17} + \frac23 m\,e^{36},\, -m\,e^{27},\, -\frac23 m\,e^{37},\, -m\,e^{47},\,  -m\,e^{57},\, -\frac23 m\,e^{67},\, 0\right)$$
%
%%-----------------  EQ 11
%
%$$\mathfrak{cp}_3^m=\left(-\frac32 m\,e^{17} + \frac12 m\,(e^{36}+e^{45}),\, -m\,e^{27},\, -\frac34 m\,e^{37},\, -\frac34 m\,e^{47},\, -\frac34 m\,e^{57},\, -\frac34 m\,e^{67},\, 0\right)$$
%%-----------------  EQ 12
%
%$$\mathfrak{cp}_4^m=\left( -\frac75 m\,e^{17} + \frac25 m\,(e^{36}+e^{45}),\, -\frac65 m\,e^{27}-\frac25 m\, e^{46},\,-\frac45 m\,e^{37},\, -\frac35 m\,e^{47},\, -\frac45 m\,e^{57},\, -\frac35 m\,e^{67},\, 0\right)$$
%
%
%%-----------------  EQ 13
%
%$$\mathfrak{cp}_5^m=\left(-\frac54 m\,e^{17}+\frac12 m\,e^{45},\, -\frac54\,m\,e^{27}-\frac12 m\,e^{46},\, -m\,e^{37},\, -\frac12 m\,e^{47},\, -\frac34 m\,e^{57},\, -\frac34 m\,e^{67},\, 0\right)$$
%
%
%%-----------------  EQ 14
%
%$$\mathfrak{cp}_6^m=$$
%
%$$\left(-\frac43 m\,e^{17} + \frac13 m\,(e^{36}+e^{45}),\, -\frac43 m\,e^{27}+\frac13 m\, (e^{35}-e^{46}),\,-\frac23 m\,e^{37} ,\, -\frac23 m\,e^{47},\, -\frac23 m\,e^{57},\, -\frac23 m\,e^{67},\, 0\right)$$
%
%
%%-----------------  EQ 15
%
%$$\mathfrak{cp}_7^m=$$
%
%$$\left(-\frac65 m\,e^{17}+\frac25 m\,e^{36},\, -\frac35 m\,e^{27},\, -\frac35 m\,e^{37},\, -\frac65 m\,e^{47}+\frac25 m\,e^{26},\, -\frac65 m\,e^{57}+\frac25 m\,e^{23},\, -\frac35 m\,e^{67},\, 0\right)$$

\end{proposition}
\begin{proof}
All the Lie algebras fulfilling the hypothesis of the theorem  are originally expressed (see the proof of \cite[Theorem 1]{CF}, expression numbers from (9) to (15)) in a basis
$\{\nu^1,\ldots,\nu^7\}$ of $\frs^*$ where the G$_2$-structure $\varphi$ adopts the following  expression:
$$
\varphi=\nu^{125}-\nu^{345}+\nu^{567}+\nu^{136}+\nu^{246}-\nu^{237}+\nu^{147}.
$$
In all the cases, $\varphi$ turns out to be locally conformal parallel with Lee 1-form $\tau = m\,\nu^7$. Now, for every Lie algebra the new
basis of 1-forms:
$$ e^1=\nu^3,\quad e^2=\nu^2,\quad e^3=\nu^1,\quad  e^4=\nu^4,\quad  e^5=- \nu^6,\quad e^6=\nu^5,\quad e^7= \nu^7,$$
expresses $\varphi$ in our canonical way \eqref{sigma}, and the structure equations of $\frak{cp}_s^m$, $1\leq s\leq 7,$ result as above.
\end{proof}

A quick inspection of the Lie algebras $\mathfrak{cp}_s^m$ listed in Proposition~\ref{algebrasFR} reveals that each of them is determined by
two tuples: one containing the eigenvalues $(\eta_1,\ldots, \eta_6)$ of the derivation $D$ and other one including the
non-identically zero structure constants $(c_{36}^1,\, c_{45}^1,\,c_{35}^2,\, c_{46}^2,\, c_{26}^4,\,c_{23}^5)$ of  the underlying
6-dimensional Lie algebra. In Table~\ref{tabla_algebras} we set both tuples for each case.

\begin{table}[h!]
\renewcommand{\arraystretch}{2.0}
\begin{center}
\begin{tabular}{|c|c|c|}
\hline
&$(\eta_1,\ldots, \eta_6)$&$(c_{36}^1,\, c_{45}^1,\,c_{35}^2,\, c_{46}^2,\, c_{26}^4,\,c_{23}^5)$\\
\hline
$\mathfrak{cp}_1^m$&$(-m, -m, -m, -m, -m, -m)$&$(0, 0, 0, 0, 0, 0)$\\
\hline
$\mathfrak{cp}_2^m$&$(-\frac{4m}{3}, -m, -\frac{2m}{3}, -m, -m, -\frac{2m}{3})$&$(\frac{2m}{3}, 0, 0, 0, 0, 0)$\\
\hline
$\mathfrak{cp}_3^m$&$(-\frac{3m}{2}, -m, -\frac{3m}{4}, -\frac{3m}{4}, -\frac{3m}{4}, -\frac{3m}{4})$&$(\frac{m}{2}, \frac{m}{2}, 0, 0, 0, 0)$\\
\hline
$\mathfrak{cp}_4^m$&$(-\frac{7m}{5}, -\frac{6m}{5}, -\frac{4m}{5}, -\frac{3m}{5}, -\frac{4m}{5}, -\frac{3m}{5})$&$(\frac{2m}{5}, \frac{2m}{5}, 0, -\frac{2m}{5}, 0, 0)$\\
\hline
$\mathfrak{cp}_5^m$&$(-\frac{5m}{4}, -\frac{5m}{4}, -m, -\frac{m}{2}, -\frac{3m}{4}, -\frac{3m}{4})$&$(0, \frac{m}{2}, 0, -\frac{m}{2}, 0, 0)$\\
\hline
$\mathfrak{cp}_6^m$&$(-\frac{4m}{3}, -\frac{4m}{3}, -\frac{2m}{3}, -\frac{2m}{3}, -\frac{2m}{3}, -\frac{2m}{3})$&$(\frac{m}{3}, \frac{m}{3}, \frac{m}{3}, -\frac{m}{3}, 0, 0)$\\
\hline
$\mathfrak{cp}_7^m$&$(-\frac{6m}{5}, -\frac{3m}{5}, -\frac{3m}{5}, -\frac{6m}{5}, -\frac{6m}{5}, -\frac{3m}{5})$&$(\frac{2m}{5}, 0, 0, 0, \frac{2m}{5}, \frac{2m}{5})$\\
\hline
\end{tabular}
\end{center}
\caption{Defining parameters of the Lie algebras $\mathfrak{cp}_s^m$.}\label{tabla_algebras}
\end{table}

Now, we shall study solutions to the Laplacian flow on every Lie algebra $\mathfrak{cp}_s^m$. As we  mentioned before,
we assume a family of $\mathrm{G}_2$-structures $\varphi(t)$ given by~\eqref{eqn:expression_sol} where the unknown data are some differentiable real functions $f_i(t)$ depending on a parameter
$t\in I\subset{\mathbb{R}}$ such that $f_i(0)=1$ and $f_i(t)\neq0$, for any $t\in I$, where $I$ is a real open interval.   Observe that in fact, the functions $f_i(t)$ are positive. The basis
$x^i(t)=f_i(t)e^i$ is adapted to the $\mathrm{G}_2$-structure at any $t$, and the structure equations for any of the Lie algebras $\mathfrak{cp}_s^m$
depend on the functions and defining parameters of the algebras contained in Table~\ref{tabla_algebras}:

\begin{equation}\label{structure-eq-x}
\begin{cases}
\begin{array}{llllll}
dx^1 &=& \dfrac{\eta_1}{f_7(t)}\,x^{17} +c_{36}^1 \dfrac{f_1(t)}{f_{36}(t)}\,x^{36} + c_{45}^1\dfrac{f_1(t)}{f_{45}(t)}\,x^{45},&\qquad dx^5& =& \dfrac{\eta_5}{f_7(t)}\,x^{57} + c_{23}^5\dfrac{f_5(t)}{f_{23}(t)}\,x^{23},\\[10pt]
dx^2 &=& \dfrac{\eta_2}{f_7(t)}\,x^{27} +c_{35}^2\dfrac{f_2(t)}{f_{35}(t)}\,x^{35} +c_{46}^2 \dfrac{f_2(t)}{f_{46}(t)}\,x^{46},&\qquad dx^6 &=& \dfrac{\eta_6}{f_7(t)}\,x^{67},\\[10pt]
dx^3& =& \dfrac{\eta_3}{f_7(t)}\,x^{37},&\qquad dx^7 &=& 0.\\[10pt]
dx^4& =& \dfrac{\eta_4}{f_7(t)}\,x^{47} +c_{26}^4 \dfrac{f_4(t)}{f_{26}(t)}\,x^{26},
\end{array}
\end{cases}
\end{equation}

Since we want to solve the LCP-flow~\eqref{LCP-flow-eq}, we need to solve two equations.  Let us start looking for necessary and
sufficient conditions on the evolution functions $f_i(t)$ in order to preserve the locally conformal parallel condition, i.e, $\varphi(t)\in \mathcal X_4$, that we state in a
 more restrictive version imposing that the Lee 1-form remains constant along the flow:
 %$\tau(t) = \tau_0$:
\begin{proposition}\label{prop-CP}
The family of $\mathrm{G}_2$-structures $\varphi(t)$ given by~\eqref{eqn:expression_sol} satisfies
\begin{equation}\label{conditions-CP}
d\varphi(t) = 3 m\,e^7\wedge \varphi(t),\quad d\ast_t\varphi(t) = 4 m\,e^7\wedge \ast_t\varphi(t),
\end{equation}
and in particular remains locally conformal parallel if and only if the evolution
functions $f_i(t)$, $1\leq i\leq 7,$ satisfy the following conditions:
\begin{itemize}
\item $\mathfrak{cp}^m_1$: For any $f_i(t)$,  $i=1,\ldots, 7$.
\item $\mathfrak{cp}^m_2$: $f_{17}(t) = f_{36}(t)$.
\item $\mathfrak{cp}^m_3$: $f_{17}(t) = f_{36}(t) = f_{45}(t)$.
\item $\mathfrak{cp}^m_4$: $f_{17}(t)= f_{36}(t) = f_{45} (t)$, $f_{27}(t)= f_{46}(t)$.
\item $\mathfrak{cp}^m_5$: $f_{17}(t)= f_{45}(t)$, $f_{27}(t) = f_{46}(t)$.
\item $\mathfrak{cp}^m_6$: $f_{17}(t) = f_{36}(t) = f_{45}(t)$, $f_{27}(t) = f_{35}(t) =  f_{46}(t)$.
\item $\mathfrak{cp}^m_7$: $f_{17}(t) = f_{36}(t)$, $\ f_{23}(t) = f_{57}(t)$, $\ f_{26}(t) = f_{47}(t)$.
\end{itemize}
\end{proposition}
\begin{proof}
  Let us start computing $d\varphi(t)$ using the general structure equations~\eqref{structure-eq-x}. Directly:
  % \begin{equation}\label{dsigma}
  % \begin{array}{lll}
  % d\varphi(t)&=& e^{1357}\left[-c_{35}^2\,f_{127}(t)+ (\eta_1+\eta_3+\eta_5)\,f_{135}(t)\right] + \\[5pt]
  % && e^{1467}\left[-c_{46}^2\,f_{127}(t)- (\eta_1+\eta_4+\eta_6)\,f_{146}(t)\right] + \\[5pt]
  % && e^{2367}\left[c_{36}^1\,f_{127}(t)- (\eta_2+\eta_3+\eta_6)\,f_{236}(t) + c_{26}^4\,f_{347}(t) + c_{23}^5\,f_{567}(t)\right] + \\[5pt]
  % && e^{2457}\left[c_{45}^1\,f_{127}(t)- (\eta_2+\eta_4+\eta_5)\,f_{245}(t)\right].
  % \end{array}
  % \end{equation}

  \begin{equation*}
  \begin{array}{lll}
  d\varphi(t)&=& x^{1357}\left[\dfrac{\eta_1+\eta_3+\eta_5}{f_7(t)}-c^2_{35}\dfrac{f_2(t)}{f_{35}(t)}\right] -
  x^{1467}\left[\dfrac{\eta_1+\eta_4+\eta_6}{f_7(t)}+c^2_{46}\dfrac{f_2(t)}{f_{46}(t)}\right] -\\[12pt]
  && x^{2367}\left[\dfrac{\eta_2+\eta_3+\eta_6}{f_7(t)}-c_{36}^1\dfrac{f_1(t)}{f_{36}(t)}-c_{26}^4\dfrac{f_4(t)}{f_{26}(t)}
  -c_{23}^5\dfrac{f_5(t)}{f_{23}(t)}\right] - x^{2457}\left[\dfrac{\eta_2+\eta_4+\eta_5}{f_7(t)}-c_{45}^1\dfrac{f_1(t)}{f_{45}(t)}\right].
  \end{array}
\end{equation*}

  Now, the equation $d\varphi(t) = 3m\,e^7\wedge \varphi(t)=\frac{3m}{f_7(t)}\,x^7\wedge \varphi(t)$ is equivalent to the following system of equations:
  \begin{equation}\label{dsigma}
  \begin{cases}
  (\eta_1+\eta_3+\eta_5+3m)\,f_{135}(t) = c_{35}^2\,f_{127}(t),\\[5pt]
  (\eta_1+\eta_4+\eta_6+3m)\,f_{146}(t) = -c_{46}^2\,f_{127}(t),\\[5pt]
  (\eta_2+\eta_4+\eta_5+3m)\,f_{245}(t) =c_{45}^1\,f_{127}(t),\\[5pt]
  (\eta_2+\eta_3+\eta_6+3m)\,f_{236}(t) = c_{36}^1\,f_{127}(t) + c_{26}^4\,f_{347}(t) + c_{23}^5\,f_{567}(t).\\[5pt]
  \end{cases}
  \end{equation}

  Similar computations for $d\ast_t\varphi(t)$ yield:
  \begin{equation*}
  \begin{array}{lll}
  d\ast_t\varphi(t)&=& -x^{12347}\left[\dfrac{\eta_1+\eta_2+\eta_3+\eta_4}{f_7(t)}-c_{23}^5\dfrac{f_5(t)}{f_{23}(t)}\right] - x^{12567}\left[\dfrac{\eta_1+\eta_2+\eta_5+\eta_6}{f_7(t)}-c_{26}^4\dfrac{f_4(t)}{f_{26}(t)}\right] - \\[12pt]
  && x^{34567}\left[\dfrac{\eta_3+\eta_4+\eta_5-\eta_6}{f_7(t)}-c_{36}^1\dfrac{f_1(t)}{f_{36}(t)}-c_{45}^1\dfrac{f_1(t)}{f_{45}(t)} -c_{35}^2\dfrac{f_2(t)}{f_{35}(t)}+c_{46}^2\dfrac{f_2(t)}{f_{46}(t)}\right].
  \end{array}
\end{equation*}

  Again, solving the equation $d\ast_t\varphi(t) = 4m\,e^7\wedge \ast_t\varphi(t) = \frac{4m}{f_7(t)} x^7\wedge\ast_t\varphi(t)$ is equivalent
  to solve the system of equations:
  \begin{equation}\label{d-ast-sigma}
  \begin{cases}
  \eta_1+\eta_2+\eta_3+\eta_4+4m= c_{23}^5\,\dfrac{f_{57}(t)}{f_{23}(t)},\\[12pt]
  \eta_1+\eta_2+\eta_5+\eta_6+4m= c_{26}^4\,\dfrac{f_{47}(t)}{f_{26}(t)},\\[12pt]
  \eta_3+\eta_4+\eta_5+\eta_6+4m= c_{36}^1\,\dfrac{f_{17}(t)}{f_{36}(t)} + c_{45}^1\,\dfrac{f_{17}(t)}{f_{45}(t)} + c_{35}^2\,\dfrac{f_{27}(t)}{f_{35}(t)} - c_{46}^2\,\dfrac{f_{27}(t)}{f_{46}(t)}.
  \end{cases}
  \end{equation}

The final result is obtained by substituting the defining parameters of the Lie algebras $\mathfrak{cp}_s^m$  listed in Table~\ref{tabla_algebras} in both the expressions~\eqref{dsigma} and \eqref{d-ast-sigma}.\end{proof}

After solving $\varphi(t)\in \mathcal X_4$, let us focus on the evolution equation $\dfrac{d\varphi(t)}{dt}=\Delta_t\varphi(t)$.  Next,  we get a generic expression of the
Laplacian $\Delta_t \varphi(t)$ suitable for any of the Lie algebras $\mathfrak{cp}_s^m$.

\begin{proposition}\label{prop-Lap}
Let $\varphi(t)$ be a family of $\mathrm{G}_2$-structures given by~\eqref{eqn:expression_sol} and remaining locally conformal parallel in the sense of~\eqref{conditions-CP}, for each Lie algebra $\mathfrak{cp}_s^m$ the Laplacian $\Delta_t \varphi(t)$ is given by:
\begin{equation*}
\Delta_t \varphi(t)=\sum_{(i,j,k)\in A\cup B} \varepsilon(i,j,k)\Delta_{ijk}\,x^{ijk} \,
\end{equation*}
where
\begin{equation}\label{laplacian2}
\begin{array}{ll}
\Delta_{127}&= \frac{m}{f^2_7}\left[3(4m + \eta_3 + \eta_4 + \eta_5 + \eta_6) + 4(\eta_1 + \eta_2) \right],\\[5pt]
 \Delta_{347}& = \frac{m}{f^2_7} \left[\frac{6m}{5}\,\delta_7 + 4 \left(\eta_3 + \eta_4\right) \right], \\[5pt]
\Delta_{567}&=\frac{m}{f^2_7} \left[ \frac{6m}{5}\,\delta_7 + 4 \left(\eta_5 + \eta_6\right) \right], \\[5pt]
\Delta_{135}&=  \frac{m}{f^2_7}\left[ \frac{4m}{3}\,\delta_6 + 3\left(\eta_2+ \eta_4+\eta_6\right)\right],\\[5pt]
\Delta_{146}& = \frac{m}{f^2_7} \left[\frac{8m}{5}\,\delta_4 + 2m\,\delta_5 + \frac{4m}{3}\,\delta_6  + 3\left(\eta_2+ \eta_3+\eta_5\right)\right],\\[5pt]
  \Delta_{236}&=\frac{m}{f^2_7}\left[\frac{8m}{3}\,\delta_2 +  2m\,\delta_3 + \frac{8m}{5}\,\delta_4 + \frac{4m}{3}\,\delta_6 + \frac{24 m}{5}\,\delta_7 + 3\left(\eta_1+ \eta_4+\eta_5\right)\right], \\[5pt]
  \Delta_{245}&=\frac{m}{f^2_7} \left[2m\,\delta_3 + \frac{8m}{5}\,\delta_4 + 2m\,\delta_5 + \frac{4m}{3}\,\delta_6 - 3\left(\eta_1+ \eta_3+\eta_6\right)\right],
\end{array}
\end{equation} and $\delta_s= \begin{cases}1,\, \text{ if }\frg\cong \mathfrak{cp}_s^m,\\ 0,\, \text{ if }\frg\ncong \mathfrak{cp}_s^m.  \end{cases}$
\end{proposition}

\begin{proof}
The Laplacian operator on 3-forms is defined as: $\Delta_t\varphi(t) = -d\ast d\ast \varphi(t) + \ast d\ast d\varphi(t)$.  Taking into account the
conformally parallel conditions~\eqref{conditions-CP} expressed in terms of the orthonormal basis $\{x_i\}_{i=1}^7$, the Laplacian of $\varphi(t)$ can be computed as:
\begin{eqnarray*}
\Delta_t\varphi(t)& = &\frac{m}{f^2_7(t)}\left[-4 \left(d\ast (x^7\wedge \ast \varphi(t))\right) + 3 \left(\ast d\ast (x^7\wedge\varphi(t))\right)\right]\\[6pt]
&=& \frac{m}{f^2_7(t)}\left[-4 \left(d\ast (x^{12347}+x^{12567}+x^{34567})\right) + 3 \left(\ast d\ast (-x^{1357}+x^{1467}+x^{2367} + x^{2457})\right)\right]\\[6pt]
&=& \frac{m}{f^2_7(t)}\left[-4 \left(d (x^{12}+x^{34}+x^{56})\right) + 3 \left(\ast d(x^{136} +x^{145}+x^{235}-x^{246})\right)\right].
\end{eqnarray*}
If we apply~\eqref{structure-eq-x} and the Hodge star operator in the second summand, we obtain the following expression %, where $\Delta \varphi(t) = \frac{m}{f_7(t)} \hat\Delta  \varphi(t)$:
\begin{equation*}\label{laplacian1}
\begin{array}{lll}
\Delta_t\varphi(t) &=& \frac{m}{f^2_7(t)} \left[ 3 \left(c_{36}^1 \dfrac{f_1(t)}{f_{36}(t)} + c_{45}^1 \dfrac{f_1(t)}{f_{45}(t)} + c_{35}^2 \dfrac{f_2(t)}{f_{35}(t)} - c_{46}^2 \dfrac{f_2(t)}{f_{46}(t)}\right) + 4 \left(\dfrac{\eta_1+\eta_2}{f_7(t)}\right) \right]  x^{127}\\[15pt]
&&+ \frac{m}{f^2_7(t)}\left[ 3\,c_{26}^4 \dfrac{f_4(t)}{f_{26}(t)} + 4 \left(\dfrac{\eta_3+\eta_4}{f_7(t)}\right) \right] x^{347} + \frac{m}{f^2_7(t)} \left[ 3\,c_{23}^5 \dfrac{f_5(t)}{f_{23}(t)} + 4 \left(\dfrac{\eta_5+\eta_6}{f_7(t)}\right) \right] x^{567}\\[15pt]
&& +\frac{m}{f^2_7(t)} \left[ 3\left(\dfrac{\eta_2+\eta_4+\eta_6}{f_{7}(t)} \right)+ 4 \, c_{35}^2 \dfrac{f_2(t)}{f_{35}(t)} \right] x^{135}  + \frac{m}{f^2_7(t)} \left[ -3\left(\dfrac{\eta_2+\eta_3+\eta_5}{f_{7}(t)} \right)+ 4 \,c_{46}^2 \dfrac{f_2(t)}{f_{46}(t)} \right] x^{146}\\[15pt]
&& +  \frac{m}{f^2_7(t)}\left[ -3 \left(\dfrac{\eta_1+\eta_4+\eta_5}{f_{7}(t)}\right) -4\left(c_{36}^1 \dfrac{f_1(t)}{f_{36}(t)}+ c_{26}^4 \dfrac{f_4(t)}{f_{26}(t)} + c_{23}^5 \dfrac{f_5(t)}{f_{23}(t)} \right) \right] x^{236}\\[15pt]
&& +  \frac{m}{f^2_7(t)}\left[ -3\left(\dfrac{\eta_1+\eta_3+\eta_6}{f_{7}(t)} \right)- 4 \, c_{45}^1 \dfrac{f_1(t)}{f_{45}(t)} \right] x^{245}.
\end{array}
\end{equation*}
To get the final expression just apply Proposition~\ref{prop-CP} together with the defining parameters of the Lie algebras collected in Table~\ref{tabla_algebras}.
\end{proof}
\begin{remark}\label{remark-coefs}
We notice that for any family of $\mathrm{G}_2$-structures $\varphi(t)$ given by~\eqref{eqn:expression_sol} and any $(i,\,j,\,k)\in A\cup B$ the expressions of $f_7^2(t)\Delta_{ijk}$ obtained in~\eqref{laplacian2}  depend only on the defining
parameters of the Lie algebra $\mathfrak{cp}_s^m$ and not on the functions $f_i(t)$.
\end{remark}

\section{Long time solutions of the Laplacian flow of an LCP $\mathrm{G}_2$-structure}
%\section{Solutions of the Laplacian flow of \textcolor{blue}{an} LCP $\mathrm{G}_2$-structure on $\mathfrak{cp}_i^m$}
 In this section we obtain long time solutions for the Laplacian flow on the solvable Lie groups $S_s$, $s=1,\ldots,7,$ where $S_s$ has underlying Lie algebra $\mathfrak{cp}_s^m$ described in Proposition \ref{algebrasFR} in terms of a basis
 $\{e^1,\dots,e^7\}$ such that the canonical 3-form $\varphi_0$ given by \eqref{sigma} is an LCP $\mathrm{G}_2$-structure.
 We divide our study  starting by the solvable Lie group $S_1$ as the results obtained on it guide the method on the rest of cases.

 \begin{theorem}\label{Thm-cp1}
Let $S_1$ be a solvable Lie group with underlying Lie algebra $\frak{cp}_1^m$. The family of $\mathrm{G}_2$-structures given by:
\begin{equation*}
  \varphi(t)=(1-4m^2t)^2\left(e^{12}+e^{34}+e^{56}\right)\wedge e^7+(1-4m^2t)^{\frac94}\left(e^{135}-e^{146}-e^{236}-e^{245}\right), \quad t\in I=\left(-\infty,\frac{1}{4m^2}\right)
\end{equation*}
is the unique solution for the Laplacian LCP-flow~\eqref{LCP-flow-eq}. Moreover, the underlying metric $g(t)$ is Einstein at any $t\in I$ and converges to a flat metric as t goes to $-\infty$.
\end{theorem}
\begin{proof}
Taking into account \eqref{structure-eq-x} and the defining parameters of the Lie algebra $\mathfrak{cp}_1^m$ given in Table \ref{tabla_algebras}, the Maurer-Cartan equations in the adapted basis $\{x^1,\dots,x^7\}$ are:

\begin{equation*}
\left\{\begin{array}{lcl}
dx^k&=&-\frac{m}{f_7}\, x^k\wedge x^7, \quad   1\leq k\leq 6,\\[4pt]
dx^7&=&0.
\end{array}\right.
\end{equation*}

Proposition~\ref{prop-CP} shows that for  $\frak{cp}_1^m$ the family of $\mathrm{G}_2$-structures $\varphi(t)$ given by~\eqref{eqn:expression_sol}
remains locally conformal parallel regardless of the evolution functions $f_i(t)$. Then, we only need to solve the evolution
equation $\dfrac{d\varphi(t)}{dt}=\Delta_t\varphi(t)$.

We get  the expression of the Laplacian $\Delta_t\varphi(t)$ substituting the defining parameters of the Lie algebra $\frak{cp}_1^m$ provided in Table~\ref{tabla_algebras} in the
generic formula  given in Proposition~\ref{prop-Lap}:
\begin{equation*}\label{Lap1}
\begin{array}{l}
\Delta_t\varphi(t)=\dfrac{-m^2 }{f^2_7(t)}\left[8\, (x^{127}+ x^{347}+x^{567})+9\, (x^{135}- x^{146}-x^{236}-x^{245})\right].
\end{array}
\end{equation*}

\noindent Now, the equalities:
\begin{equation*}
\Delta_{127}=\Delta_{347}=\Delta_{567}=\frac{-8m^2}{f_7^2(t)}, \quad \Delta_{135}=\Delta_{146}=\Delta_{236}= \Delta_{245}=\frac{-9m^2}{f_7^2(t)},
\end{equation*}

\noindent imply respectively by Lemma~\ref{lem:equalities2} part i) that $f_{12}=f_{34}=f_{56}$ and $f_{135}=f_{146}=f_{236}= f_{245}$. From the first group we get $f_4(t)=\frac{f_{12}(t)}{f_3(t)}$ and
$f_6(t)=\frac{f_{12}(t)}{f_5(t)}$, thus, substituting in the second one we get $f_1^2(t)=f_2^2(t)=f_3^2(t)=f_5^2(t)$. Furthermore, as $f_i(t)>0$, we conclude that $f_i(t)=f(t)$ for any $1\leq i\leq 6$.

At this point, solving the evolution equation \eqref{eq:evol_eq_expand1} reduces to solve the following system of two differential equations with unknowns $f(t)$ and $f_7(t)$:

\begin{equation*}
\begin{cases}
\begin{array}{lll}
\dfrac{-8 m^2}{f_7^2(t)}=\Delta_{127}=\dfrac{f_{127}'}{f_{127}} &=& \dfrac{d}{dt} \ln (f_{127}) =  \dfrac{d}{dt} [\ln (f_{1}(t)) +  \ln (f_{2}(t)) + \ln (f_{7}(t))]=2\dfrac{f'(t)}{f(t)}+\dfrac{f'_7(t)}{f_7(t)},\\[10pt]
\dfrac{-9 m^2}{f_7^2(t)}=\Delta_{135}=\dfrac{f_{135}'}{f_{135}} &=& \dfrac{d}{dt} \ln (f_{135}) =  \dfrac{d}{dt} [\ln (f_{1}(t)) +  \ln (f_{3}(t)) + \ln (f_{5}(t))]=3\dfrac{f'(t)}{f(t)},
\end{array}
\end{cases}
\end{equation*}
%
%
%\begin{equation*}
%\begin{cases}
%\begin{array}{lll}
%\dfrac{-8 m^2}{f_7^2(t)} &=& \dfrac{2\,f'(t)}{f(t)}  + \dfrac{f'_7(t)}{f_7(t)}\\[10pt]
%\dfrac{-9 m^2}{f_7^2(t)} &=& \dfrac{3\,f'(t)}{f(t)},
%\end{array}
%\end{cases}
%\end{equation*}
which is equivalent to:
\begin{equation*}
\begin{cases}
\begin{array}{lll}
\dfrac{-2 m^2}{f_7^2(t)} &=& \dfrac{f'_7(t)}{f_7(t)},\\[10pt]
\dfrac{-3 m^2}{f_7^2(t)} &=& \dfrac{f'(t)}{f(t)}.
\end{array}
\end{cases}
\end{equation*}
The first equation involves only $f_7(t)$ and can be explicitly solved: $$f_7(t) f_7'(t) = -2m^2\Longrightarrow f_7(t) = \left(-4m^2 t + C\right)^{\nicefrac12}.$$
Moreover, using the fact that $f_7(0) = 1$, we get that $C=1$ and $f_7(t) = \left(1-4m^2 t\right)^{\nicefrac12}$.  With this value for $f_7(t)$,
it is also possible to solve explicitly the second equation:
$$\dfrac{-3 m^2}{1-4m^2 t} =\dfrac{f'(t)}{f(t)}\Longrightarrow \frac34 \ln (1-4m^2t) = \ln f(t) + C.$$
Again, the value of $C$ is determined imposing the initial condition $f(0)=1$, obtaining that $f(t) = (1-4m^2t)^{\nicefrac34}$. The domains of the functions
$f(t)$ and $f_7(t)$ imply that the family $\varphi(t)$ of $\mathrm{G}_2$-structures is defined for any $t\in I=(-\infty,\frac{1}{4m^2})$.

Concerning the metric, it turns out that the non-vanishing components  of the  curvature tensor
$R_{ijkl}=g(R(x_i,x_j)x_k,x_l)$ at any $t\in I$ are (modulo its symmetry properties):
\begin{equation*}
  R_{ijji}=-\frac{m^2}{1-4 m^2 t} \quad \text{ for any } 1\leq i <  j \leq 7.
\end{equation*}
Thus, $\displaystyle\lim_{t \rightarrow  -\infty} R( g_t)=0$. Moreover, an standard computation shows that the Ricci tensor  $Ric(g_t)_{ij}=\sum_{k=1}^7R_{kijk}$ satisfies
\begin{equation*}
Ric(g_t)=-\frac{6 m^2}{1-4m^2t} \, g_t,
\end{equation*}
that is, $g_t$ is Einstein concluding the proof.
\end{proof}

For the rest of the Lie algebras $\frak{cp}_s^m$, we obtain the following explicit solutions:
\begin{theorem}\label{sol-flujo}
Let $S_s$ be a solvable Lie group with underlying Lie algebra $\frak{cp}_s^m$. The family of $\mathrm{G}_2$-structures given below is a solution for the Laplacian flow:
\begin{itemize}
  \item $\frak{cp}_2^m$: For $t\in (-\infty,\frac{3}{10m^2})$,
\begin{equation*}
  \varphi(t)=(1-\frac{10}{3}m^2t)^{\frac{11}{5}}\left(e^{127}-e^{236}\right)+
  (1-\frac{10}{3}m^2t)^{2}\left(e^{347}-e^{567}\right)+
  (1-\frac{10}{3}m^2t)^{\frac{12}{5}}\left(e^{135}-e^{146}-e^{245}\right).
\end{equation*}
\medskip
%%%%%%%%%%%%%%%%%%%%
\item $\frak{cp}_3^m$:  For $t\in (-\infty,\frac{1}{3m^2})$,
\begin{equation*}
\varphi(t)=(1-3m^2t)^{\frac{7}{3}}(e^{127} -e^{236}-e^{245})+(1-3m^2t)^{2}(e^{347}+e^{567})+(1-3m^2t)^{\frac{5}{2}}(e^{135}-e^{146}).
\end{equation*}
\medskip
%%%%%%%%%%%%%%%%%%%%%%
\item $\frak{cp}_4^m$:  For $t\in (-\infty,\frac{5}{14m^2})$,
\begin{equation*}
\varphi(t)=(1-\frac{14}{5}m^2t)^{\frac{17}{7}}(e^{127}-e^{146}-e^{236}-e^{245})+(1-\frac{14}{5}m^2t)^{2}(e^{347}+e^{567})+
(1-\frac{14}{5}m^2t)^{\frac{18}{7}}e^{135}.
\end{equation*}
\medskip
%%%%%%%%%%%%%%%%%%%%%%%%
\item $\frak{cp}_5^m$:  For $t\in (-\infty,\frac{1}{3m^2})$,
\begin{equation*}
\varphi(t)=(1-3m^2t)^{\frac{4}{3}}(e^{127}-e^{146}-e^{245})+(1-3m^2t)^{2}(e^{347}+e^{567})+
(1-3m^2t)^{\frac{5}{2}}(e^{135}-e^{236}).
\end{equation*}
\medskip
%%%%%%%%%%%%%%%%%%%%%%%%
\item $\frak{cp}_6^m$:  For $t\in (-\infty,\frac{3}{8m^2})$,
\begin{equation*}
\varphi(t)=(1-\frac{8}{3}m^2t)^{\frac{5}{2}}(e^{127} + e^{135}-e^{146}-e^{236}-e^{245})+(1-\frac{8}{3}m^2t)^{2}(e^{347}+e^{567}).
\end{equation*}\medskip
%%%%%%%%%%%%%%%%%%%%%%%%%%%%
\item $\frak{cp}_7^m$:  For $t\in (-\infty,\frac{5}{14m^2})$,
\begin{equation*}
\varphi(t)=(1-\frac{14}{5}m^2t)^{\frac{15}{7}}(e^{127}+e^{347}+e^{567} - e^{236})+(1-\frac{14}{5}m^2t)^{\frac{18}{7}}(e^{135}-e^{146}-e^{245}).
\end{equation*}

\end{itemize}
% Moreover, the underlying metric $g(t)$ is Laplacian soliton and converges to a flat metric as t goes to $-\infty$.
\end{theorem}
\begin{proof}
  Inspired by the solution to the Laplacian flow on the solvable Lie group $S_1$ obtained in Theorem~\ref{Thm-cp1} we will consider families of G$_2$-structures
  of type~\eqref{eqn:expression_sol} on the rest of Lie groups $S_s$ where the evolution functions $f^s_i(t)$ are specifically given by:
  \begin{equation}\label{funcionesG}
  f^s_i(t)=(1-\alpha_s\,m^2 \, t)^{\beta_i^s},\quad \text{for any }1\leq i\leq 7,
  \end{equation}
  hence, for each Lie algebra $\frak{cp}_s^m$, the unknowns are now $\alpha_s\in\mathbb R^*$ and $\beta^s_i\in\mathbb R$,  $i = 1, \dots, 7$.
  Now, Proposition~\ref{prop-CP} states necessary and sufficient conditions for the property being LCP to be preserved during the flow, which under
  the assumption~\eqref{funcionesG} transform into relations involving the $\beta_i$ coefficients as follows:
  % \begin{itemize}
  % \item $\mathfrak{cp}^m_2$: $\beta_1+\beta_7=\beta_3+\beta_6$.
  % \item $\mathfrak{cp}^m_3$: $\beta_1+\beta_7=\beta_3+\beta_6=\beta_4+\beta_5$.
  % \item $\mathfrak{cp}^m_4$: $\beta_1+\beta_7=\beta_3+\beta_6=\beta_4+\beta_5$, $\beta_2+\beta_7=\beta_4+\beta_6$.
  % \item $\mathfrak{cp}^m_5$: $\beta_1+\beta_7=\beta_4+\beta_5$, $\beta_2+\beta_7=\beta_4+\beta_6$.
  % \item $\mathfrak{cp}^m_6$: $\beta_1+\beta_7=\beta_3+\beta_6=\beta_4+\beta_5$, $\beta_2+\beta_7=\beta_3+\beta_5=\beta_4+\beta_6$
  % \item $\mathfrak{cp}^m_7$: $\beta_1+\beta_7=\beta_3+\beta_6$, $\beta_2+\beta_3=\beta_5+\beta_7$, $\beta_2+\beta_6=\beta_4+\beta_7$.
  % \end{itemize}
  \begin{equation}\label{betaLCP}
  \begin{array}{l}
     \mathfrak{cp}^m_2: \beta_1+\beta_7=\beta_3+\beta_6;\\[5pt]
     \mathfrak{cp}^m_3: \beta_1+\beta_7=\beta_3+\beta_6=\beta_4+\beta_5;\\[5pt]
     \mathfrak{cp}^m_4: \beta_1+\beta_7=\beta_3+\beta_6=\beta_4+\beta_5,\quad \beta_2+\beta_7=\beta_4+\beta_6;\\[5pt]
     \mathfrak{cp}^m_5: \beta_1+\beta_7=\beta_4+\beta_5,\quad  \beta_2+\beta_7=\beta_4+\beta_6;\\[5pt]
     \mathfrak{cp}^m_6: \beta_1+\beta_7=\beta_3+\beta_6=\beta_4+\beta_5,\quad  \beta_2+\beta_7=\beta_3+\beta_5=\beta_4+\beta_6;\\[5pt]
     \mathfrak{cp}^m_7: \beta_1+\beta_7=\beta_3+\beta_6,\quad \beta_2+\beta_3=\beta_5+\beta_7,\quad  \beta_2+\beta_6=\beta_4+\beta_7.
  \end{array}
  \end{equation}

In addition, system \eqref{eq:evol_eq_expand1}  reduces to
  \begin{equation*}
 \Delta_{ijk} =  \frac{(f_{ijk})'}{f_{ijk}}=\frac{-\alpha m^2(\beta_i+\beta_j+\beta_k)}{(1-\alpha m^2 t)},
  \end{equation*}
 where the unknowns are $\alpha$ and $\beta_1,\ldots, \beta_7$.

  %requires the expresion of the 3-form:
%  \begin{equation*}
%  \varphi(t)=\sum_{(i,j,k)\in A\cup B}\epsilon_{ijk} x^{ijk}=
%  \sum_{(i,j,k)\in A\cup B}\epsilon_{ijk} (1-\alpha m^2t)^{\beta_i+\beta_j+\beta_k}e^{ijk}
%  \end{equation*}
%  at any $t$, where we remind the sets of indices $A=\{(1,2,7),(1,3,5),(3,4,7),(5,6,7)\}$ and $B=\{(1,4,6), (2,3,6),(2,4,5)\}$. Hence,
%  \begin{eqnarray*}
%    \dfrac{d}{d t}\varphi(t) &=&
%    -\alpha m^2 \sum_{(i,j,k)\in A\cup B}\epsilon_{ijk}
%   (\beta_i+\beta_j+\beta_k) (1-\alpha m^2t)^{\beta_i+\beta_j+\beta_k-1}e^{ijk}\\
%     &=& \frac{-\alpha m^2}{1-\alpha m^2t}  \sum_{(i,j,k)\in A\cup B}\epsilon_{ijk}
%    (\beta_i+\beta_j+\beta_k)x^{ijk}
%  \end{eqnarray*}
Explicitly, taking the $\Delta_{ijk}$ coefficients given in Proposition~\ref{prop-Lap}:
\begin{eqnarray*}
\beta_1+\beta_2+\beta_7&=&-\frac{(1-\alpha m^2t)^{1-2\beta_7}}{\alpha m}\left[3(4m + \eta_3 + \eta_4 + \eta_5 + \eta_6) + 4(\eta_1 + \eta_2) \right],\\[5pt]
\beta_3+\beta_4+\beta_7&=&-\frac{(1-\alpha m^2t)^{1-2\beta_7}}{\alpha m}\left[\dfrac{6m}{5}\,\delta_7 + 4 \left(\eta_3 + \eta_4\right) \right],\\[5pt]
\beta_5+\beta_6+\beta_7&=&-\frac{(1-\alpha m^2t)^{1-2\beta_7}}{\alpha m}\left[ \dfrac{6m}{5}\,\delta_7 + 4 \left(\eta_5 + \eta_6\right) \right],\\[5pt]
\beta_1+\beta_3+\beta_5&=&-\frac{(1-\alpha m^2t)^{1-2\beta_7}}{\alpha m}\left[ \dfrac{4m}{3}\,\delta_6 + 3\left(\eta_2+ \eta_4+\eta_6\right)\right],\\[5pt]
\beta_1+\beta_4+\beta_6&=&-\frac{(1-\alpha m^2t)^{1-2\beta_7}}{\alpha m}\left[\dfrac{8m}{5}\,\delta_4 + 2m\,\delta_5 + \dfrac{4m}{3}\,\delta_6  + 3\left(\eta_2+ \eta_3+\eta_5\right)\right],\\[5pt]
\beta_2+\beta_3+\beta_6&=&-\frac{(1-\alpha m^2t)^{1-2\beta_7}}{\alpha m}\left[\dfrac{8m}{3}\,\delta_2 +  2m\,\delta_3 + \dfrac{8m}{5}\,\delta_4 + \dfrac{4m}{3}\,\delta_6 + \dfrac{24 m}{5}\,\delta_7 + 3\left(\eta_1+ \eta_4+\eta_5\right)\right],\\[5pt]
\beta_2+\beta_4+\beta_5&=&-\frac{(1-\alpha m^2t)^{1-2\beta_7}}{\alpha m}\left[2m\,\delta_3 + \dfrac{8m}{5}\,\delta_4 + 2m\,\delta_5 + \dfrac{4m}{3}\,\delta_6 - 3\left(\eta_1+ \eta_3+\eta_6\right)\right].
\end{eqnarray*}
Clearly, the latter system admits solution only if $\beta_7=\frac12$. Now, for each Lie algebra $\frak{cp}_s^m$ the values of $\alpha$ and $\beta_1,\ldots,\beta_6$ result of solving
the system that yields substituting above the concrete values of $\eta_1,\ldots,\eta_6$ listed in Table~\ref{tabla_algebras} joint with the corresponding relations~\eqref{betaLCP} involving the
preservation of the LCP condition during the flow. The values of the solution parameters are listed in Table~\ref{parametros} and the resulting solutions $\varphi(t)$ are picked in the
statement of the theorem.
\begin{table}[h!]
\renewcommand{\arraystretch}{1.8}
\begin{center}
\begin{tabular}{|c|c|c||c|c|c|}
\hline
Lie group& $\alpha$ & $(\beta_{1},\ldots, \beta_7)$&Lie group& $\alpha$ & $(\beta_{1},\ldots, \beta_7)$\\
\hline
%$\mathfrak{cp}_2^m$&$\frac{10}{3}$&$\big(\frac{9}{10}, \frac45,\frac{7}{10},\frac45,\frac45,\frac{7}{10},\frac12\big)$\\[6pt]
$S_2$&$\frac{10}{3}$&$\big(\frac{9}{10}, \frac45,\frac{7}{10},\frac45,\frac45,\frac{7}{10},\frac12\big)$&$S_5$&$ 3$&$\big(\frac{11}{12}, \frac{11}{12},\frac56,\frac23,\frac34,\frac34,\frac12\big)$\\
\hline
%$\mathfrak{cp}_3^m$&$ 3$&$\big(1, \frac56,\frac34,\frac34,\frac34,\frac34,\frac12\big)$\\[6pt]
$S_3$&$ 3$&$\big(1, \frac56,\frac34,\frac34,\frac34,\frac34,\frac12\big)$&$S_6$&$ \frac{8}{3}$&$\big(1, 1,\frac34,\frac34,\frac34,\frac34,\frac12\big)$\\
\hline
%$\mathfrak{cp}_4^m$&$ \frac{14}{5}$&$\big(1, \frac{13}{14},\frac{11}{14},\frac57,\frac{11}{14},\frac57,\frac12\big)$\\[6pt]
$S_4$&$ \frac{14}{5}$&$\big(1, \frac{13}{14},\frac{11}{14},\frac57,\frac{11}{14},\frac57,\frac12\big)$&$S_7$&$ \frac{14}{5}$&$\big(\frac{13}{14}, \frac{10}{14},\frac{10}{14},\frac{13}{14},\frac{13}{14},\frac{10}{14},\frac12\big)$\\
\hline
\end{tabular}
\end{center}
\caption{Defining parameters of the functions $f_i(t)=(1-\alpha m^2t)^{\beta_i}$. }\label{parametros}
\end{table}
\end{proof}

\begin{remark}\label{RiccF}

The non-vanishing elements (up to symmetries) of the Riemannian curvature of $g_t$ are tabulated
 in Tables~\ref{Tabla_Curvatura} and~\ref{Table3} in Appendix, where the coefficients $C_1,\,C_2,\,C_3,\,C_4$ are given in terms of the parameter $\alpha$ by:
$$C_1=\frac{-6m^2}{1-\alpha\, m^2 t},\quad C_2 = \frac{-m^2}{3}\left(\frac{1}{1-\alpha\, m^2 t}\right),\quad C_3 = \frac{-m^2}{4}\left(\frac{1}{1-\alpha\, m^2 t}\right),\quad C_4 = \frac{-m^2}{5}\left(\frac{1}{1-\alpha\, m^2 t}\right).$$
A direct computation of the Ricci tensors in the orthonormal basis $\{x_i(t)\}_{i=1}^7$ shows that:
\begin{equation*}
\begin{array}{l}
  \frak{cp}_2^m : \text{Ric}(g_t)=C_2\, diag(22, 17, 12, 17, 17, 12, 17), \\[5pt]
  \frak{cp}_3^m : \text{Ric}(g_t)=C_3\,  diag(32, 22, 17, 17, 17, 17, 22), \\[5pt]
  \frak{cp}_4^m : \text{Ric}(g_t)=C_4\,  diag(37, 32, 22, 17, 22, 17, 27), \\[5pt]
  \frak{cp}_5^m : \text{Ric}(g_t)=C_3\, diag(27, 27, 22, 12, 17, 17, 22), \\[5pt]
  \frak{cp}_6^m : \text{Ric}(g_t)=C_2\,  diag(21, 21, 11, 11, 11, 11, 16), \\[5pt]
  \frak{cp}_7^m : \text{Ric}(g_t)=C_4\, diag(32, 17, 17, 32, 32, 17, 27),
  \end{array}
\end{equation*}
Hence, unlike the case of $S_1$, none of the metrics listed in Theorem~\ref{sol-flujo} obtained as solutions of the Laplacian LCP-flow for $S_s$, $s=2,\ldots, 7$, are Einstein.
\end{remark}

\begin{proposition} The $\mathrm G_2$-structures obtained in Theorem~\ref{Thm-cp1} and~\ref{sol-flujo} are Laplacian solitons of shrinking type.
\end{proposition}
\begin{proof}
Recall that a $\mathrm G_2$-structure $\varphi$ is called \emph{Laplacian soliton} if it satisfies the equation $$\Delta\varphi= \lambda \varphi+\mathcal L_X \varphi,$$ for some real number $\lambda$ and some vector field $X$.  Depending on the sign of $\lambda$, Laplacian solitons are called \emph{shrinking} (if $\lambda<0$); \emph{steady} (if $\lambda=0$) or \emph{expanding} (if $\lambda>0$).

In the left-invariant setting, the Lie derivative of a 3-form $\Omega$ can be computed following the formula:
\begin{equation*}
\mathcal L_X \Omega(Y_1,Y_2,Y_3)=-\Omega([X,Y_1],Y_2,Y_3)-\Omega(Y_1,[X,Y_2],Y_3)-
\Omega(Y_1,Y_2,[X,Y_3]),
\end{equation*}
where $Y_1,\,Y_2,\,Y_3$ are invariant vector fields.

In our case, consider the left-invariant vector field $X = -\frac{m}{f_7(t)}X_7$, where $X_7$ denotes the dual of the
1-form $x^7$. Then, taking into account the generic structure equations~\eqref{structure-eq-x} and the formula above we get an expression of
the Lie derivative $\mathcal L_X \varphi(t)$ in  terms of the defining parameters $\eta_1,\dots,\eta_6$ of the Lie algebras:
\begin{equation}\label{Lie-dev}
\begin{array}{lll}
\mathcal L_X \varphi(t)&=&\frac{m}{f_7^2(t)}\left[(\eta_1+\eta_2)\,x^{127} + (\eta_3+\eta_4)\,x^{347} + (\eta_5+\eta_6)\,x^{567}+(\eta_1+\eta_3+\eta_5)\,x^{135}\right.\\[7pt]
&&\qquad \left.-(\eta_1+\eta_4+\eta_6)\,x^{146}-(\eta_2+\eta_3+\eta_6)\,x^{236}-(\eta_2+\eta_4+\eta_5)\,x^{245}\right].
\end{array}
\end{equation}

Now, if we compute $\Delta\varphi(t)-\mathcal L_X \varphi(t)$, using~\eqref{Lie-dev}, Proposition~\ref{prop-Lap} and Table~\ref{tabla_algebras}, we obtain:
$$
\Delta\varphi(t)-\mathcal L_X \varphi(t) = -\frac{m^2}{f^2_7(t)}\varphi(t)\left(6\,\delta_1 + 5\,\delta_2 + \frac92\,\delta_3+\frac{21}{5}\,\delta_4 + \frac92\,\delta_5 + 4\,\delta_6 + \frac{21}{5}\,\delta_7\right),
$$ so they are Laplacian solitons.  Moreover, since the constant $\lambda_i$ is negative in all cases, the solitons are of shrinking type.
\end{proof}

\section{Long time solutions of the Laplacian coflow of an LCP $\mathrm{G}_2$-structure}

In this section we consider the Laplacian coflow~\eqref{LCP-coflow-eq}. More concretely, we seek explicit solutions to the coflow on the set of solvable
Lie groups endowed with a left-invariant LCP structure listed in Proposition~\ref{algebrasFR}.

We look for solutions within the families of invariant $\mathrm{G}_2$-structures $\varphi(t)$ of type~\eqref{eqn:expression_sol} depending on some
unknown functions $f_i(t)$, $1\leq i\leq 7$, in the same terms as it has been set in the paper. Then, at any $t\in I$, the
4-form $\psi(t)=\ast_t\varphi(t)$ involved in the evolution equation of the coflow can be expressed in terms of the adapted basis $\{x^i(t)\}_{i=1}^7$ as:
\begin{equation*}\label{codeformation}
  \begin{array}{lll}
    \psi(t)&=&x^{3456}+x^{1256}+x^{1234}-x^{2467}+x^{2357}+x^{1457}+x^{1367}\\[7pt]

    &=&f_{3456}(t)e^{3456}+f_{1256}(t)e^{1256}+f_{1234}(t)e^{1234}-f_{2467}(t)e^{2467}\\[7pt]
    &&+f_{2357}(t)e^{2357}+f_{1457}(t)e^{1457}+f_{1367}(t)e^{1367}.
  \end{array}
\end{equation*}
and the Laplacian of the 4-form $\psi(t)$:
\begin{equation}\label{eq:delta4ijk}
\Delta_t\psi(t)=\sum_{(l,m,n,o)\in \mathcal{K}\cup\{(2,4,6,7)\}}\varepsilon(l,m,n,o)\Delta_{lmno}\,x^{lmno}+
\sum_{(l,m,n,o)\notin \mathcal{K}\cup\{(2,4,6,7)\}}\Delta_{lmno}x^{lmno}.
\end{equation}
where $\mathcal K = \{(1,2,3,4), (1,2,5,6), (1,3,6,7), (1,4,5,7), (2,3,5,7), (3,4,5,6)\}$. The symbols $\varepsilon(l,m,n,o)$ are defined as:
\begin{equation*}
\varepsilon(l,m,n,o)=\begin{cases}
1&\text{ if }(l,m,n,o)\in  \mathcal{K},\\
-1&\text{ if }(l,m,n,o)=(2,4,6,7);\\
\end{cases}
\end{equation*}
Therefore, by a similar a similar argument as in the flow case, the first equation of the Laplacian LCP-coflow~\eqref{LCP-coflow-eq} becomes the system of differential equations:
\begin{gather*}\label{coeq-flujo}
\begin{cases}
\begin{array}{ll}
\Delta_{lmno}=\frac{-(f_{lmno})'}{f_{lmno}}, \qquad\,\quad \,\, &\text{ if }\,  (l,m,n,o)\in\mathcal{K}\cup\{(2,4,6,7)\},\\
\Delta_{lmno}=0, \qquad\,\quad \,\, &\text{ otherwise.}
\end{array}
\end{cases}
\end{gather*}

Concerning the preservation of the LCP property of $\ast_t\psi(t)$, functions $f_i(t)$ must satisfy the relations stated in Proposition~\ref{prop-CP}.
Now, in the next result we show that there is a correspondence between solutions of the ansatz type~\eqref{eqn:expression_sol} of the Laplacian LCP-flow and
the Laplacian LCP-coflow assuming that the functions are of potential type.

\begin{theorem}\label{Thm:flujo-coflujo}
Let $\varphi(t)$ and $\widetilde\varphi(t)$ be two different families of $\mathrm{G}_2$-structures on $\mathfrak{cp}^m_s$ with $s=1,\ldots, 7$, given by~\eqref{eqn:expression_sol}, where
%\begin{equation*}
%\varphi(t)=\sum_{(ijk)\in \mathcal{I}} x^{ijk} -  \sum_{(ijk)\in \mathcal{J}} x^{ijk}
%\end{equation*}
%with $\mathcal{I}=\{(127),(347),(567),(135)\}$ and $\mathcal{J}=\{(146),(236),(245)\}$ where $x^i(t)=a_i(t)e^i$ with
%\begin{equation*}
%f_i(t)=a_i(t)=(1 -\alpha\,m^2 t)^{\beta_i}, \quad \beta_7=\frac12,\qquad\text{and}\qquad \widetilde{f_i}(t) = b_i(t)=(1-\gamma\,  m^2 t)^{\delta_i}, \quad \delta_7=\frac12.
%\end{equation*}
\begin{equation*}
f^s_i(t)=(1 -\alpha\,m^2 t)^{\beta_i}\qquad\text{and}\qquad \widetilde{f^s_i}(t)=(1-\gamma\,  m^2 t)^{\delta_i}, \quad \text{for}\quad i=1,\ldots, 7,
\end{equation*}
and $\beta_7=\frac12$ and $\delta_7=\frac12$. If the defining parameters of the functions $f^s_i(t)$ and $\widetilde{f^s_i}(t)$ are related by:

\begin{equation}\label{relation}
  \gamma = \alpha\left( \frac{2-\sum_{i=1}^{7}\beta_i}{2}\right),\quad \text{and}\quad
\delta_i = \frac12+\frac{1-2 \beta_i}{-2+\sum_{j=1}^{7}\beta_j}  \quad \text{ with } \quad i\in \{1, \dots, 7\},
\end{equation}
then:

\begin{enumerate}
\item[(i)] $\varphi(t)$ is LCP if and only if $\widetilde \varphi(t)$ is LCP.
\item[(ii)] $\varphi(t)$ solves the Laplacian LCP-flow~\eqref{LCP-flow-eq} if and only if $\widetilde\psi(t) = \widetilde\ast_{t} \widetilde\varphi(t)$ solves the Laplacian LCP-coflow~\eqref{LCP-coflow-eq}.
\end{enumerate}
\end{theorem}

\begin{proof}
 We denote by $(l,m,n,o) \in \mathcal{K}$ the set of complementary indexes to $(i,j,k) \in A\cup B$ (see page~\pageref{eq:delta3ijk}), i.e. $(l,m,n,o)=\widehat{(i,j,k)} = (1,\dots, \hat{i}, \dots, \hat{j}, \dots, \hat{k}, \dots, 7 )$. With this notation, \eqref{relation} implies
\begin{equation}\label{relacion-importante}
\gamma \, (\delta_l + \delta_m + \delta_n + \delta_o) = -\alpha \, (\beta_i+\beta_j+\beta_k),
\end{equation}
for all $(l,m,n,o) \in \mathcal{K}$ and $(l,m,n,o) = \widehat{(i,j,k)}$.

Now we prove the two statements of the theorem.

% \begin{enumerate}
% \item[(i)]
\smallskip

\noindent (i) Let us consider two pair of indexes $(i_1,j_1,k_1), (i_2,j_2,k_2) \in A\cup B$ such that they have a common index, let us say $k_1=k_2=k$. Under this hypothesis and making use of~\eqref{relacion-importante}
the following identities hold:
\begin{eqnarray*}
\gamma (\delta_1 + \ldots + \delta_7) &= & \gamma (\delta_1 + \ldots + \delta_7)\\
\gamma(\delta_{i_1} + \delta_{j_1} + \delta_{k}) + \gamma (\delta_{l_1}+ \delta_{m_1} + \delta_{n_1} + \delta_{o_1})  &= & \gamma(\delta_{i_2} + \delta_{j_2} + \delta_{k}) + \gamma (\delta_{l_2}+ \delta_{m_2} + \delta_{n_2} + \delta_{o_2})\\
\gamma(\delta_{i_1} + \delta_{j_1} + \delta_{k}) - \alpha (\beta_{i_1}+ \beta_{j_1} + \beta_{k})  &= & \gamma(\delta_{i_2} + \delta_{j_2} + \delta_{k}) - \alpha (\beta_{i_2}+ \beta_{j_2} + \beta_{k})\\
\gamma[(\delta_{i_2} + \delta_{j_2}) - (\delta_{i_1} + \delta_{j_1})]  &=& -\alpha[(\beta_{i_1} + \beta_{j_1}) - (\beta_{i_2} + \beta_{j_2})].
\end{eqnarray*}
We observe that the relations~\eqref{betaLCP} concerning the preservation  of the LCP property are always of the form:
\begin{equation*}
\beta_{i_1} + \beta_{j_1} = \beta_{i_2} + \beta_{j_2},
\end{equation*}
for a pair of indexes satisfying that $(i_1,j_1,k)  \in A,\, (i_2,j_2,k)\in  B$. The LCP conditions for $\widetilde\varphi(t)$ are exactly the same that the LCP conditions for $\varphi(t)$ interchanging the parameters $\beta_i$ for $\delta_i$. Therefore, considering the non-zero values of the parameter $\alpha$ of the solutions
of the Laplacian flow (see Theorem~\ref{sol-flujo}) it is easy to see that $\gamma\neq0$ in all the cases, hence we conclude
that $\varphi(t)$ is LCP if and only if $\widetilde\varphi(t)$  is so.

\noindent (ii) Let $\varphi(t)$ and $\widetilde\varphi(t)$ be two families of $\mathrm{G}_2$-structures~\eqref{eqn:expression_sol} whose defining parameters of the functions $f^s_i(t)$ and $\widetilde f^s_i(t)$ are related by~\eqref{relation}.
Let $\varphi(t)$ be a solution of the Laplacian LCP-flow; we want to prove that $\widetilde\psi(t)=\widetilde\ast_t\widetilde\varphi(t)$ is
a solution of the corresponding coflow. That $\tilde\psi(t)$ is solution of the coflow is equivalent to:
\begin{equation*}
  (\widetilde f^s_7(t))^2\tilde \Delta_{lmno}= \gamma m (\delta_l+\delta_m+\delta_n+\delta_o)\qquad \text{ for any }(l,m,n,o)\in \mathcal{K}\cup\{(2,4,6,7)\},
  \end{equation*}
 where we have used the same ideas as in the proof of Theorem~\ref{sol-flujo}, that is, the functions $\widetilde f^s_i(t)$ are of potential type and $\delta_7=\frac12$.

Firstly observe that, as the Hodge star operator commutes with the Laplacian operator, we have that
$\Delta_t\ast_t\varphi(t)=\ast_t\Delta_t\varphi(t)$, hence, the coefficients of $\Delta_t\varphi(t)$ and $\Delta_t\psi(t)$ appearing in the linear combinations~\eqref{eq:delta3ijk} and \eqref{eq:delta4ijk}
 are related by:
\begin{equation*}
\Delta_{lmno}=\Delta_{ijk},\qquad\text{ for any }(i,j,k)\in A\cup B \text{ and } (l,m,n,o)=\widehat{(i,j,k)}.
\end{equation*}
This fact together with  the non-dependence of the $(f^s_7(t))^2\Delta_{ijk}$ with respect to the specific chosen functions noticed in
Remark~\ref{remark-coefs} yields:
\begin{equation*}
  (\widetilde f^s_7(t))^2\tilde\Delta_{lmno}=(\widetilde f^s_7(t))^2\tilde\Delta_{ijk}= (f^s_7(t))^2\Delta_{ijk},\qquad\text{ for any }(i,j,k)\in A\cup B \text{ and } (l,m,n,o)=\widehat{(i,j,k)}.
\end{equation*}
Now, since $\varphi(t)$ is solution of the flow then we have:
\begin{equation*}
(f^s_7(t))^2\Delta_{ijk}=- \alpha m (\beta_i+\beta_j+\beta_k)\qquad \text{ for any }(i,j,k)\in A\cup B.
  \end{equation*}
Therefore, bearing in mind~\eqref{relacion-importante} the following sequence of identities hold:

$$(\widetilde f^s_7(t))^2\widetilde{\Delta}_{lmno} = (f^s_7(t))^2\Delta_{ijk}%\stackrel{\eqref{flow}}{=}
 =-\alpha m (\beta_i+\beta_j+\beta_k)= \gamma m (\delta_l+\delta_m+\delta_n+\delta_o),$$
for every $(l,m,n,o)\in\mathcal{K}\cup\{(2,4,6,7)\}$, that is, $\widetilde\psi(t)$ is a solution of the coflow.

The converse of the statement is basically the same and we omit the proof.
\end{proof}

As a consequence of the previous theorem, for every Lie group $S_s$ we provide in the following corollary  an explicit solution of the LCP-coflow based on the defining parameters of the solutions of the flow contained in Table~\ref{parametros}.%Theorem~\ref{sol-flujo}.
%We also include the information concerning the Ricci tensor of the induced metrics??

\begin{corollary}\label{cor:coflow}
Let $S_s$ be a solvable Lie group with underlying Lie algebra $\frak{cp}_s^m$. The family of $\mathrm{G}_2$-structures given below is solution for the Laplacian coflow:
\begin{itemize}
  \item $\frak{cp}_1^m$:  For $t\in(-\frac{1}{6 m^2},\,\infty)$,
\begin{equation*}
  \varphi(t)=(1+6m^2t)^{\frac{7}{6}}\left(e^{127}+e^{347}+e^{567}\right)+
  (1+6m^2t)\left(e^{135}-e^{146}-e^{236}-e^{245}\right).
\end{equation*}
\medskip
%%%%%%%%%%%%%%%%%%%%%%%%%%
  \item $\frak{cp}_2^m$:  For $t\in(-\frac{3}{16 m^2},\,\infty)$,
\begin{equation*}
  \varphi(t)=(1+\frac{16}{3}m^2t)^{\frac{17}{16}}\left(e^{127}-e^{236}\right)+
  (1+\frac{16}{3}m^2t)^{\frac{19}{16}}\left(e^{347}+e^{567}\right)+
  (1+\frac{16}{3}m^2t)^{\frac{15}{16}}\left(e^{135}-e^{146}-e^{245}\right).
\end{equation*}
\medskip
%%%%%%%%%%%%%%%%%%%%%%%%%%
\item $\frak{cp}_3^m$:  For $t\in(-\frac{1}{5 m^2},\,\infty)$,
\begin{equation*}
\varphi(t)=(1+5m^2t)(e^{127}-e^{236}-e^{245})+(1+5m^2t)^{\frac{6}{5}}(e^{347}+e^{567})+(1+5m^2t)^{\frac{9}{10}}(e^{135}-e^{146}).
\end{equation*}
\medskip
%%%%%%%%%%%%%%%%%%%%%%%%%%
\item $\frak{cp}_4^m$: For $t\in(-\frac{5}{24 m^2},\,\infty)$,
\begin{equation*}
\varphi(t)=(1+\frac{24}{5}m^2t)^{\frac{23}{24}}(e^{127}-e^{146}-e^{236}-e^{245})+(1+\frac{24}{5}m^2t)^{\frac{29}{24}}(e^{347}+e^{567})+
(1+\frac{24}{5}m^2t)^{\frac{7}{8}}e^{135}.
\end{equation*}
\medskip
%%%%%%%%%%%%%%%%%%%%%%%%%%
\item $\frak{cp}_5^m$:  For $t\in(-\frac{1}{5 m^2},\,\infty)$,
\begin{equation*}
\varphi(t)=(1+5m^2t)(e^{127}-e^{146}-e^{245})+(1+5m^2t)^{\frac{6}{5}}(e^{347}+e^{567})+
(1+5m^2t)^{\frac{9}{10}}(e^{135}-e^{236}).
\end{equation*}
\medskip
%%%%%%%%%%%%%%%%%%%%%%%%%%
\item $\frak{cp}_6^m$:  For $t\in(-\frac{14}{3 m^2},\,\infty)$,
\begin{equation*}
\varphi(t)=(1+\frac{3}{14}m^2t)^{\frac{13}{14}}(e^{127}+e^{135}-e^{146}-e^{236}-e^{245})+(1+\frac{3}{14}m^2t)^{\frac{17}{14}}(e^{347}+e^{567}).
\end{equation*}
\medskip
%%%%%%%%%%%%%%%%%%%%%%%%%%
\item $\frak{cp}_7^m$:  For $t\in(-\frac{5}{24 m^2},\,\infty)$,
\begin{equation*}
\varphi(t)=(1+\frac{24}{5}m^2t)^{\frac{9}{8}}(e^{127}+e^{347}+e^{567}-e^{236})+(1+\frac{24}{5}m^2t)^{\frac{7}{8}}(e^{135}-e^{146}-e^{245}).
\end{equation*}
\end{itemize}
%Moreover, the underlying metric $g(t)$ is Laplacian soliton and converges to a flat metric as t goes to $-\infty$.
\end{corollary}

\begin{remark}
Direct computations show that, for each Lie algebra $\mathfrak{cp}^m_s$, the Ricci tensors of the metrics $\widetilde g_t$ induced by the solutions $\widetilde \varphi(t)$ to the coflow given in Corollary ~\ref{cor:coflow} are given by
 $$\text{Ric}(\widetilde g_t)=\left(\frac{1-\alpha\,m^2t}{1+\gamma\,m^2t}\right)\text{Ric}(g_t),$$
where $\text{Ric}(g_t)$ are the Ricci tensors of the metrics $g_t$ induced by the solutions $\varphi(t)$ to the flow given in Remark~\ref{RiccF}.  Thus, as in the flow case, only the solutions on $S_1$ are Einstein.

\end{remark}

\section*{Acknowledgments}
\noindent
The authors would like to thank Anna Fino and Luis Ugarte for useful comments on the subject.  This work has been partially supported by the projects  MTM2017-85649-P
(AEI/FEDER, UE), and E22-17R ``\'Algebra y Geometr\'ia" (Gobierno de Arag\'on/FEDER).
%MINECO (Spain) MTM2014-54804-P, MTM2014-58616-P, and Gobierno de Arag\'on/Fondo Social Europeo--Grupo Consolidado E15 Geometr\'{\i}a.  
The third author would also like to thank the Fields Institute for its support during her stay in Toronto.

%%%%%%%%%%%%%%%%%%%%%%%%%%%%%%%%%%%%%%%%%%%%%%%%%%%%%%%%%%%%%%%%%%%%%%%%%%%%%%%%%%%%%%
%BIBLIOGRAPHY%%%%%%%%%%%%%%%%%%%%%%%%%%%%%%%%%%%%%%%%%%%%%%%%%%%%%%%%%%%%%%%%%%%%%%%%
%%%%%%%%%%%%%%%%%%%%%%%%%%%%%%%%%%%%%%%%%%%%%%%%%%%%%%%%%%%%%%%%%%%%%%%%%%%%%%%%%%%%%%

\newpage

\section*{Appendix}

\begin{table}[h!]
\begin{center}
\begin{tabular}{|c|l|}
\hline
Solvmanifold & \hspace{4.5cm} $R(g_t)$ \hspace{1cm} \\
\hline
&\\
&$R_{1367}(t)=R_{1637}(t)=\frac{-2}{3}C_2,\quad R_{1717}(t)=\frac{-16}{3}\,C_2,\quad R_{1736}(t)=\frac{4}{3}C_2,  $\\[5pt]
& $R_{1212}(t)=R_{1414}(t)=R_{1515}(t)=-4\,C_2,   $\\[5pt]
$S_2$& $R_{1313}(t)=R_{1616}(t)=R_{3636}(t)=\frac{-7}{3}\,C_2,\quad R_{3737}(t)=R_{6767}(t)=\frac{-4}{3}\,C_2,$\\[5pt]
&$R_{2323}(t)=R_{2626}(t)=R_{3434}(t)=R_{3535}(t)=R_{4646}(t)=R_{5656}(t)=-2\,C_2, $\\[5pt]
&$ R_{2424}(t)=R_{2525}(t)=R_{2727}(t)=R_{4545}(t)=R_{4747}(t)=R_{5757}(t)=-3\,C_2.$\\
&\\
\hline
&\\
&$R_{1212}(t) = \frac{-3}{2}\,C_3,\quad R_{1313}(t) = R_{1414}(t) = R_{1515}(t) = R_{1616}(t) = \frac{-17}{4}\,C_3, $\\[5pt]
&$R_{1367}(t) = R_{1457}(t) = -R_{1547}(t) = -R_{1637}(t)  = -R_{3716}(t) = -R_{4715}(t) =$\\[5pt]
&$R_{5714}(t) = R_{6713}(t) = \frac{-3}{4}\,C_3,\quad R_{3645}(t) = R_{4536}(t) = \frac{-C_3}{2},$\\[5pt]
$S_3$&$R_{1717}(t) = -9\,C_3,\quad R_{1736}(t) = R_{1745}(t) = R_{3617}(t) = R_{4517}(t) =\frac{3}{2}\,C_3,$\\[5pt]
&$R_{2323}(t) =  R_{2424}(t) = R_{2525}(t) = R_{2626}(t) = R_{3636}(t) =R_{4545}(t) = -3\,C_3, $\\[5pt]
&$R_{3456}(t) =  -R_{3546}(t) = -R_{4635}(t) = R_{5634}(t) =  \frac{C_3}{4}, \quad R_{2727}(t) = -4\,C_3,$\\[5pt]
&$R_{3434}(t) = R_{3535}(t) = R_{3737}(t) = R_{4646}(t)  = R_{4747}(t) = R_{5656}(t) =$\\[5pt]
&$R_{5757}(t) = R_{6767}(t) = \frac{-9}{4}\,C_3.$\\[5pt]
\hline
&\\
&$R_{1234}(t) = R_{1256}(t) = -R_{1423}(t) = -R_{1625}(t)  = -R_{2314}(t) = -R_{2516}(t) =$\\[5pt]
&$R_{3412}(t) = R_{3456}(t) = -R_{3546}(t) = -R_{4635}(t)  = R_{5612}(t) = R_{5634}(t) =\frac{C_4}{5},$\\[5pt]
&$R_{1367}(t) = -R_{1547}(t) = -R_{2467}(t) = R_{2647}(t)  = -R_{4715}(t) = R_{4726}(t) =$\\[5pt]
&$R_{6713}(t) = -R_{6724}(t) =\frac{-3}{5}\,C_4,\quad R_{1313}(t) = R_{1515}(t) =\frac{-27}{5}\,C_4,$\\[5pt]
&$R_{1414}(t) = R_{1616}(t) =-4\,C_4,\quad R_{1457}(t) = -R_{1637}(t) =\frac{-4}{5}\,C_4,$\\[5pt]
$S_4$&$R_{1736}(t) =R_{1745}(t) =R_{3617}(t) =R_{4517}(t) = \frac{7}{5}\,C_4,\quad R_{1212}(t) =\frac{-42}{5}\,C_4,$\\[5pt]
&$R_{2323}(t) = R_{2525}(t) =\frac{-24}{5}\,C_4,\quad R_{2424}(t) = R_{2626}(t) =\frac{-17}{5}\,C_4,$\\[5pt]
&$R_{2746}(t) = R_{4627}(t) =\frac{-6}{5}\,C_4,\quad R_{3434}(t) = R_{4646}(t) =R_{5656}(t) =\frac{-12}{5}\,C_4,$\\[5pt]
&$R_{3535}(t) = R_{3737}(t) = R_{5757}(t) =\frac{-16}{5}\,C_4,\quad R_{3636}(t) = R_{4545}(t) =-3\,C_4,$\\[5pt]
&$R_{3645}(t) = R_{4536}(t) = \frac{-2}{5}\,C_4,\quad R_{3716}(t) = -R_{5714}(t) =\frac45\,C_4,$\\[5pt]
&$R_{4747}(t) = R_{6767}(t) = \frac{-9}{5}\,C_4,\quad R_{1717}(t) = \frac{-49}{5}\,C_4.$\\[5pt]
\hline
\end{tabular}
\end{center}
\caption{Non-vanishing coefficients of the curvature of the metric $g_t$ induced by the solutions of the LCP flow expresssed in the
adapted basis $\{x_i\}_{i=1}^7$.}\label{Tabla_Curvatura}
\end{table}

\begin{table}[h!]
\begin{center}
\begin{tabular}{|c|l|}
\hline
\hspace{1cm}  \hspace{1cm} & \hspace{1cm} $R(g_t)$ \hspace{1cm} \\
\hline
&\\
&$R_{1212}(t) = R_{1717}(t) = R_{2727}(t) = \frac{-25}{5}\,C_3,\quad R_{1313}(t) = R_{2323}(t) = -5\,C_3,$\\[5pt]
&$R_{1256}(t) = -R_{1625}(t) = -R_{2516}(t) = R_{5612}(t) = \frac{C_3}{4},$\\[5pt]
&$R_{1414}(t) = R_{2424}(t) = R_{4545}(t) =  R_{4646}(t) = R_{5656}(t) = R_{5757}(t) = $\\[5pt]
&$R_{6767}(t) = \frac{-9}{4}\,C_3,\quad R_{1515}(t) = R_{2626}(t) = \frac{-7}{2}\,C_3,$\\[5pt]
$S_5$&$R_{1457}(t) = -R_{2467}(t) = R_{5714}(t) =  -R_{6724}(t) =  \frac{-3}{4}\,C_3,\quad R_{3434} = -2\,C_3,$\\[5pt]
&$R_{1547}(t) = -R_{2647}(t) = R_{4715}(t) =  -R_{4726}(t) =  \frac{C_3}{2},\quad R_{3737} = -4\,C_3,$\\[5pt]
&$R_{1745}(t) = -R_{2746}(t) = R_{4517}(t) =  -R_{4627}(t) =  \frac{5}{4}\,C_3,\quad R_{4747} = -\,C_3,$\\[5pt]
&$R_{1616}(t) = R_{2525}(t) = \frac{-15}{4}\,C_3,\quad R_{3535}(t) = R_{3636}(t) = -3\,C_3.$\\[5pt]
\hline
&\\
&$R_{1234}(t) = R_{1256}(t) = R_{3412}(t) =  R_{3456}(t) = R_{5612}(t) = R_{5634}(t) = \frac{C_2}{6},$\\[5pt]
&$R_{1313}(t) = R_{1414}(t) = R_{1515}(t) =  R_{1616}(t) = R_{2323}(t) = R_{2424}(t) = R_{2525}(t) = $\\[5pt]
&$R_{2626}(t) = \frac{-31}{12}\,C_2,\quad R_{1212}(t) = R_{1717}(t) = R_{2727}(t) =\frac{-16}{3}\,C_2,$\\[5pt]
&$R_{1324}(t) = -R_{1423}(t) = R_{1526}(t) =  -R_{1625}(t) = -R_{2314}(t) = R_{2413}(t) = R_{2615}(t) = $\\[5pt]
&$R_{3546}(t) = -R_{3645}(t) = -R_{4536}(t) = R_{4635}(t) = \frac{C_2}{12}, $\\[5pt]
$S_6$&$R_{1367}(t) = R_{1457}(t) = -R_{1547}(t) =  -R_{1637}(t) = R_{2357}(t) =  -R_{2467}(t) = -R_{2537}(t) =$\\[5pt]
&$ R_{2647}(t) =-R_{3716}(t) = -R_{3725}(t) = -R_{4715}(t) =  R_{4726}(t) = R_{5714}(t) =R_{5723}(t) = $\\[5pt]
&$ R_{6713}(t) =-R_{6724}(t) = \frac{-C_2}{3},$\\[5pt]
&$R_{1736}(t) = R_{1745}(t) =  R_{2735}(t) =-R_{2746}(t) = R_{3527}(t) = R_{3617}(t) =  R_{4517}(t) =$\\[5pt]
&$ -R_{4627}(t)  = \frac{2}{3}\,C_2,\quad R_{3535}(t) = R_{3636}(t) =  R_{4545}(t) =R_{4646}(t) =\frac{-19}{12}\,C_2,$\\[5pt]
&$R_{3434}(t) = R_{3737}(t) =  R_{4747}(t) =R_{5656}(t) =R_{5757}(t) =R_{6767}(t) =\frac{-4}{3}\,C_2,$\\[5pt]
\hline
&\\
&$R_{1234}(t) = R_{1256}(t) =  -R_{1423}(t) =R_{1526}(t) = -R_{2314}(t) = R_{2615}(t) =  R_{3412}(t) =$\\[5pt]
&$ R_{3456}(t) =-R_{3645}(t) = -R_{4536}(t) =  R_{5612}(t) =R_{5634}(t) =\frac{-18}{5}\,C_4,$\\[5pt]
&$ R_{1313}(t) =R_{1616}(t) = R_{2424}(t) =  R_{2525}(t) =R_{3535}(t) = R_{4646}(t) =\frac{-17}{5}\,C_4,$\\[5pt]
&$R_{1367}(t) = -R_{1637}(t) =  -R_{2467}(t) =-R_{2537}(t) = R_{2735}(t) = -R_{2746}(t) =  R_{3527}(t) =$\\[5pt]
$S_7$&$-R_{3716}(t) =-R_{3725}(t) = -R_{4627}(t) =  R_{6713}(t) =-R_{6724}(t) =\frac{-3}{5}\,C_4,$\\[5pt]
&$ R_{1414}(t) =R_{1515}(t) = R_{1717}(t) =  R_{4545}(t) =R_{4747}(t) = R_{5757}(t) =\frac{-36}{5}\,C_4,$\\[5pt]
&$ R_{1736}(t) =R_{2357}(t) = R_{2647}(t) =  R_{3617}(t) =R_{4726}(t) = R_{5723}(t) =\frac{6}{5}\,C_4,$\\[5pt]
&$R_{2323}(t) =R_{2626}(t) = R_{3636}(t) =  \frac{-12}{5}\,C_4,\quad R_{2727}(t) =R_{3737}(t) = R_{6767}(t) =  \frac{-9}{5}\,C_4.$\\[5pt]
\hline
\end{tabular}
\end{center}
\caption{Continuation of Table~\ref{Tabla_Curvatura}.}\label{Table3}
\end{table}

\bigskip

%\small\noindent Universidad del Pa\'{\i}s Vasco, Facultad de Ciencia y Tecnolog\'{\i}a, Departamento de Matem\'aticas,
%Apartado 644, 48080 Bilbao, Spain. \\
%\texttt{marisa.fernandez@ehu.es}\\
%\texttt{victormanuel.manero@ehu.es}

\bigskip

%\small\noindent Dipartimento di Matematica, Universit\`a di
%Torino, Via Carlo Alberto 10, Torino, Italy.\\
%\texttt{annamaria.fino@unito.it}

\end{document}